\documentclass[12pt,oneside]{amsart}

\usepackage{geometry}                
\usepackage{graphicx}
\usepackage{amssymb}
\usepackage{epstopdf}
\usepackage{color}
\usepackage{bbm, dsfont,hyperref}

\usepackage{ulem}
\newtheorem{theorem}{Theorem}
\newtheorem{lemma}{Lemma}[section]
\newtheorem{corollary}{Corollary}
\newtheorem{proposition}{Proposition}[section]
\newtheorem{conjecture}{Conjecture}
\theoremstyle{definition}
\newtheorem{definition}{Definition}
\theoremstyle{remark}
\newtheorem{remark}{Remark}

\theoremstyle{example}

\usepackage{amsmath}
\usepackage{mathrsfs}

\usepackage{verbatim}
\usepackage{amsmath,amssymb,amsthm}             
\usepackage{amssymb}             
\usepackage{amsfonts}            
\usepackage{mathrsfs}            
\usepackage{amsthm}
\usepackage{algorithm}  
\usepackage{algorithmicx}  
\usepackage{algpseudocode}  
\usepackage{mathtools}
\usepackage{commath}
\usepackage{physics}
\usepackage{bm}
\usepackage{graphicx}

\usepackage{float}
\usepackage{listings}
\usepackage{subfigure}
\usepackage{multirow}
\usepackage{diagbox}
\usepackage{color}
\usepackage{bbm}
\usepackage[export]{adjustbox}
\usepackage{enumerate}
\usepackage{bbm}

\usepackage{amsmath}
\usepackage{mathrsfs}

\newcommand{\RNum}[1]{\uppercase\expandafter{\romannumeral #1\relax}}

\begin{document}

\title{On the exact orders of critical value in Finitary Random Interlacements}

\author{Zhenhao Cai}
\address[Zhenhao Cai]{Peking University}
\email{caizhenhao@pku.edu.cn}

\author{Yuan Zhang}
\address[Yuan Zhang]{Peking University}
\email{zhangyuan@math.pku.edu.cn}
\urladdr{https://www.math.pku.edu.cn/teachers/zhangyuan/eindex.html}
	\maketitle

\tableofcontents

\begin{abstract}
	In this paper, we prove the exact orders of critical intensity $u_*(T)$ in Finitary Random Interlacements (FRI) in $\mathbb{Z}^d, \ d\ge 3$ with respect to the expected fiber length $T$. We show that as $T\to\infty$, 
	$$
	u_*(T)\sim \left\{
	\begin{aligned}
	&T^{-1}, \hspace{0.36 in} d\ge 5\\
	&T^{-1}\log T, \ d=4\\
	&T^{-1/2}, \hspace{0.25 in} d=3
	\end{aligned}
	\right..
	$$
	Our estimates also give the order of magnitude at which the percolative phase transition with respect to $T$ takes place. 
\end{abstract}

\section{Introduction}

Finitary Random Interlacements (FRI) Model was introduced by Bowen in \cite{bowen2019finitary} as a killed version of the well-known Random Interlacements (RI) proposed by Sznitman \cite{sznitman2010vacant}, composed of geometrically killed simple random walks. Finitary Random Interlacements, denoted by $\mathcal{FI}^{u,T}$, on a certain graph $G$ can be characterized by its model parameters. The multiplicative parameter $u$ governs the intensity of trajectories and is parallel to the intensity parameter in RI. And for the geometrically killing parameter $T$, it gives the average length of each killed simple random walk trajectory, while at the same time $T+1$ is inversely proportional to the intensity of trajectories. The precise definition(s) of FRI can be found in \cite{bowen2019finitary, procaccia2021percolation} as well as in Section \ref{section_notation} of this paper.

Thus, for any finite subset $A$ of $G$, the aforementioned tradeoff mechanism implies that for fixed $u$ as $T$ increases, all trajectories traversing $A$ tends to start from, and end at some ``galaxies far far away", making the ``local behavior" of FRI resemble that of the original RI model. This was made rigorous in \cite{bowen2019finitary}, and we refer to \cite{cai2021non} for more heuristic discussions. At a macroscopic scale, Bowen proved in \cite{bowen2019finitary} the existence of infinite clusters in FRI (as an edge set) for sufficiently large $T$ when $G$ is non-amenable. For three or higher dimensional lattice, \cite{procaccia2021percolation} showed that FRI has a non-trivial percolative phase transition with respect to $T$. However, as a result of the tradeoff mechanism, FRI has been shown in \cite{cai2021non}  to be globally non-monotonic for $T\in (0,\infty)$, making the global existence and uniqueness of a critical  killing parameter $T_c$ remain open.

On the other hand, note that FRI is always monotone with respect to the intensity factor $u$, a non-trivial critical percolation intensity $u_*(T)$ was proved in \cite{prevost2020percolation} to exist for all sufficiently small $T$ and in \cite{10.1214/21-ECP424} for all $T\in (0,\infty)$. And for the asymptotic of $u_*(T)$, we proved in \cite{10.1214/21-ECP424} that for $T\ll 1$,  FRI is more Bernoulli like and  $u_*(T)\cdot T\to -\log (1-p_d^c)/2$, where $p_d^c$ is the critical probability for Bernoulli bond percolation in $\mathbb{Z}^d$. While for the opposite end as $T\to \infty$, the phase transition had been shown to occur between some polynomial orders of $T$, and multiple non-sharp bounds for such orders were given in \cite{cai2021non,10.1214/21-ECP424}. It was further conjectured in \cite{cai2021non} that $u_*(T)\sim T^{-1}$ for all $d\ge 3$.

In this paper, we prove the exact order of the asymptotic of $u_*(T)$ as $T\to\infty$. For $d\ge 5$, we validate that $u_*(T)$ is indeed inversely proportional with respect to $T$. However, for $d=4$, a logarithm correction term is needed. And for $d=3$, $u_*(T)\sim T^{-1/2}$.

\begin{remark}
	We now see the asymptotic orders found in this paper more ``natural" than what was previously guessed, as they are actually in consistence with (and also determined by) the order of simple random walk intersection probabilities (Theorem 3.3.2, \cite{lawler2013intersections}), and the capacity estimations on the simple random walk trajectories (Theorem 2 in \cite{jain1968range}; Corollary 1.4, Proposition 1.5 in \cite{asselah2018capacity}). One may see the roles played by both of them in the subsequent proof. However, the fact that all random walks in this paper are killed geometrically induces substantial difficulties in estimating the killed-capacities. See Proposition \ref{lemma_capT}. 
\end{remark}

A different but highly relevant model was considered by Erhard and Poisat in \cite{MR3519252}, where they studied a Poisson cloud of Wiener sausages of deterministic length in $\mathbb{R}^d, \ d\ge 4$. When they considered the percolative phase transition with respect to the length and radius of the sausages, highly parallel asymptotic behaviors were observed. The model of Wiener sausages and finitary random interlacements are intensively correlated, and are both related to the recently introduced Worm Model, as discussed in \cite{worm2021}. Moreover, although the proofs in this paper were derived independently, we remark that some general ideas in Section \ref{section4.3} and \ref{section_upperbound} of this paper, share the spirits of Section 3-4, \cite{MR3519252}. However, it is worth noting that main results in these two papers do not follow from one another. And they are subsequently different in terms of the technicalities tackled, such as the estimations on the killed-capacities, and the criteria for ``good boxes" in the coarse-graining argument in Section \ref{section_upperbound} of this paper.

\section{Notations and preliminaries}\label{section_notation}

In this section, we introduce notations that will be later used in this paper. At the same time, we also cite certain preliminary results about FRI from previous studies, especially \cite{bowen2019finitary,cai2021continuity,10.1214/21-ECP424,procaccia2021percolation}. The notations adopted in this paper are generally in consistence with our previous works, say \cite{10.1214/21-ECP424,procaccia2021percolation} in this series of study. 

\textbf{Graph $\left(\mathbb{Z}^d,\mathbb{L}^d\right)$ and common metrics:} Let $\mathbb{Z}^d$ be the $d$-dimensional lattice. We denote the $l^{\infty}$ and Euclidean distances on $\mathbb{Z}^d$ by $|\cdot|$ and $|\cdot|_2$ respectively. We also denote by $\mathbb{L}^d$ the set of undirected edges on $\mathbb{Z}^d$ (i.e., $\mathbb{L}^d:=\left\lbrace e=\{x,y\}:x,y\in \mathbb{Z}^d\ \text{such that}\ |x-y|_2=1 \right\rbrace $). For any non-empty subsets $A,B\subset \mathbb{Z}^d$, we define the distance between $A$ and $B$ by $d(A,B):=\min_{x\in A,y\in B}|x-y|$.

\textbf{Lexicographical orders for vertices and sets of vertices:} For $x_1,x_2\in \mathbb{Z}^d$, say $x_1$ is lexicographically-smaller than $x_2$ (denoted by $x_1\vartriangleleft x_2$) if there exists some $1\le i\le d$ such that $x_1^{(i)}<x_2^{(i)}$ (where $x^{(i)}$ is the $i$-th coordinate of $x$) and for each $1\le j<i$, $x_1^{(j)}=x_2^{(j)}$. 

Moreover, for any $A=\{x_1,...,x_m\},B=\{y_1,...,y_n\}\subset \mathbb{Z}^d$ (assume that $x_1\vartriangleleft...\vartriangleleft x_m$ and $y_1\vartriangleleft...\vartriangleleft y_n$), if $A\subset B$, or $A\not\subset B$ while $x_{i_0}\vartriangleleft y_{i_0}$, where $i_0=\min\{i:x_{i_0}\neq y_{i_0}\}$, then we also say that $A$ is lexicographically-smaller than $B$ and write $A\vartriangleleft B$.

\textbf{Boundaries of a set of vertices:} For any non-empty subset $A\subset \mathbb{Z}^d$, we define the inner boundary of $A$ by $\partial A:=\left\lbrace x\in A:\exists\ z\in \mathbb{Z}^d\setminus A\ \text{s.t.}\ |x-z|_2=1 \right\rbrace $ and define the outer boundary of $A$ by $\partial^{out}A:=\left\lbrace x\in \mathbb{Z}^d\setminus A:\exists\ z\in A\ \text{s.t.}\ |x-z|_2=1 \right\rbrace $.  

\textbf{Boxes in $\mathbb{Z}^d$:} For any $x\in \mathbb{Z}^d$ and integer $n\ge 1$, we denote the box with size $n$ and corner $x$ by $B_x(n):=\left( x+\left[ 0,n \right)^d\right)\cap \mathbb{Z}^d$; we also denote that $\widetilde{B}_x(n):=\left( x+\left[ -n,2n \right)^d\right)\cap \mathbb{Z}^d$. Meanwhile, for any $T\ge 1$, let $n_T:=\lfloor \sqrt{T}\rfloor $, $B_x^T:=B_x(n_T)$ and $\widetilde{B}_x^T:=\widetilde{B}_x(n_T)$.

\textbf{Random walks in $\mathbb{Z}^d$ ($d\ge 3$):} For any $x\in \mathbb{Z}^d$, we denote the law of simple random walks starting from $x$ by $P_x\left(\cdot\right)$. Moreover, we also consider the geometrically killed random walks: for any $x\in \mathbb{Z}^d$ and $T>0$, let $X_\cdot\sim P_x\left(\cdot\right)$ and $N_T\sim Geo\left(\frac{1}{T+1}\right)$ (i.e., for any integer $k\ge 0$, $P\left(N_T=k\right)=\frac{1}{T+1} \left(\frac{T}{T+1} \right)^k$), which are independent to each other; then $\left\lbrace X_i\right\rbrace_{0\le i\le N_T}$ is called as a geometrically killed random walk starting from $x$ with average length $T$. We denote its law by $P_x^{(T)}$.

\textbf{Hitting time and entrance time:} For any random walk $\left\lbrace X_i\right\rbrace_{0\le i<m}$ ($m$ can be either finite or infinite) and subset $A\subset \mathbb{Z}^d$, we denote the first time $X_\cdot$ hitting $A$ by $H_A\left( X_\cdot\right):=\min\left\lbrace 0\le i<m:X_i\in A \right\rbrace$ (we always set that $\min \emptyset=\infty$). Meanwhile, we also denote the entrance time by $\bar{H}_A\left( X_\cdot\right):=\min\left\lbrace 1\le i<m:X_i\in A \right\rbrace$.

\textbf{Escaping probability and capacity:} For any subset $A\subset \mathbb{Z}^d$ and $x\in A$, define the escaping probability on $A$ starting from $x$ w.r.t. simple random walks as $Es_A(x):=P_x\left(\bar{H}_A=\infty\right)$. Note that for any internal vertex $y\in A\setminus \partial A$, $Es_A(x)=0$. Similarly, one may also define the escaping probability on $A$ starting from $x$ w.r.t. geometrically killed random walks with average length $T$ by $Es_A^{(T)}:=P_x^{(T)}\left(\bar{H}_A=\infty \right)$. It is worth pointing out that for each $y\in A\setminus \partial A$, $Es_A^{(T)}(y)=P_y^{(T)}\left(N_T=0 \right)=\frac{1}{T+1}$, and that for all $x\in A$, $Es_A^{(T)}(x)\ge Es_A(x)$ always holds.

Consider the capacity w.r.t. simple random walks for any finite subset $A\subset \mathbb{Z}^d$: $cap(A):=\sum_{x\in A}Es_A(x)$. By Proposition 6.5.2 in \cite{lawler2010random}, there exist $c_1(d),c_2(d)>0$ such that for all $n\ge 1$, 
\begin{equation}\label{cap_box}
	c_1 n^{d-2}\le cap\left(B_0(n)\right)\le  c_2 n^{d-2}. 
\end{equation}
According to \cite{bowen2019finitary}, we also consider the $T$-capacity: $cap^{(T)}(A):=\sum_{x\in A}Es_A^{(T)}(x)$. Since $Es_A^{(T)}(x)\ge Es_A(x)$ for all $x\in A$, one has $cap^{(T)}(A)\ge cap(A)$.

\textbf{Edge sets and paths:} We denote by $W^{\left[0,\infty \right) }$ the set of all nearest-neighbor paths with finite lengths on $\mathbb{Z}^d$. Precisely, each element of $W^{\left[0,\infty \right)}$ is an array of vertices $(x_0,...,x_n)$ such that for all $0\le i\le n-1$, $|x_i-x_{i+1}|_2=1$. For each $\eta=(x_0,...,x_n)\in W^{\left[0,\infty \right) }$, we say the length of $\eta$ is $n$. Note that each path $\eta=(x_0,x_1,...,x_n)$ with length $n \ge 1$ can be regarded as a unique edge set $\left\lbrace \{x_i,x_{i+1}\}\right\rbrace_{0\le i\le n-1}$. Hence, we no longer distinguish the notations between a path and its corresponding edge set in the rest of this paper. Similarly, for each $\mathcal{A}=\sum_{i\in \mathcal{I}}\delta_{\eta_i}$ (a point measure on $W^{\left[0,\infty \right) }$), we also, without causing further confusion, equate $\mathcal{A}$ and the edge set $\cup_{i\in \mathcal{I}}\eta_i$. 


For any edge set $\mathcal{E} \subset \mathbb{L}^d$, we write the set of all vertices covered by $\mathcal{E}$ as $V\left(\mathcal{E}\right)$ (i.e., $V\left(\mathcal{E}\right)=\cup_{\{x,y\}\in \mathcal{E}}\{x,y\}$). 


\textbf{Definitions of finitary random interlacements (FRI):} We denote the Lebesgue measure on $ \left[0,\infty \right)$ by $\lambda^+$. Let $v^{(T)}:=\sum_{x\in \mathbb{Z}^d}\frac{2d}{T+1}P_x^{(T)}\left(\cdot\right)$. Note that $ \lambda^+\times v^{(T)}$ is a $\sigma$-finite measure on $ \left[0,\infty \right)\times W^{\left[0,\infty \right) } $. Then according to \cite{bowen2019finitary}, the FRI is defined as follows:
\begin{definition}\label{def_FRI1}
	For any $T>0$, let $\mathcal{FI}^{T}=\sum_{i\in \mathbb{N}}\delta_{(u_i,\eta_i)}$ be the Poisson point process on $ \left[0,\infty \right)\times W^{\left[0,\infty \right) } $ with intensity measure $ \lambda^+\times v^{(T)}$. For any $u>0$, finitary random interlacements with expected fiber length $T$ and level $u$ are defined as
	\begin{equation}
		\mathcal{FI}^{u,T}:=\sum_{i\in \mathbb{N}:u_i\le u}\delta_{\eta_i}. 
	\end{equation}
\end{definition}
An alternative definition of FRI introduced in \cite{procaccia2021percolation} is sometimes useful as well:

\begin{definition}\label{def_FRI2}
	For any $u,T>0$, define an independent sequence of Poisson random variables $\left\lbrace N_x \right\rbrace_{x\in \mathbb{Z}^d}\overset{i.i.d.}{\sim}Pois\left(\frac{2du}{T+1}\right) $. For each $x\in \mathbb{Z}^d$, sample $N_x$ geometrically killed random walks with law $P_x^{(T)}$ independently. Then the finitary random interlacements with expected fiber length $T$ and level $u$ is the point measure consisting of all trajectories sampled above. 
\end{definition}

According to Definition \ref{def_FRI2}, for any finite subset $K\subset \mathbb{Z}^d$, the number of paths in $\mathcal{FI}^{u,T}$ starting from $K$ has the distribution of $Pois\left(\frac{2du}{T+1}\cdot |K| \right)$, where $|K|$ is the cardinality of $K$.

\textbf{FRI traversing a finite set:} The part of FRI $\mathcal{FI}^{u,T}$ traversing a finite subset $K\subset \mathbb{Z}^d$ can be discribed as follows. We first introduce a truncation mapping $\pi_K$: for each path $\eta=(x_0,...,x_n)$, if $V\left( \eta\right)\cap K=\emptyset$, let $\pi_K\left(\eta\right)=\emptyset$; otherwise, let $\pi_K\left(\eta\right)=(x_m,...,x_n)$, where $m=\min\left\lbrace 0\le i\le n:x_i\in K  \right\rbrace $. Then we denote the point measure composed of the parts of trajectories in $\mathcal{FI}^{u,T}$ after hitting $K$ by
\begin{equation}
		\mathcal{FI}^{u,T}_K:=\sum_{\eta \in \mathcal{FI}^{u,T}:V\left( \eta\right) \cap K\neq \emptyset}\delta_{\pi_K\left(\eta \right) }. 
\end{equation}
According to Lemma 2.8 in \cite{procaccia2021percolation}, for each $x\in K$, the number of paths in $\mathcal{FI}^{u,T}_K$ starting from $x$ is $Pois\left(2du\cdot Es^{(T)}_K(x)\right)$; meanwhile, all these paths have the law $P_x^{(T)}$ and are independent to each other. As a corollary, the number of paths intersecting $K$ is $Pois\left(2du\cdot cap^{(T)}(K) \right)$.

\textbf{Connection between sets of vertices:} For any sets of vertices $A,B\subset \mathbb{Z}^d$ and edge set $\mathcal{E}\subset \mathbb{L}^d$, we say that $A$ and $B$ are connected by $\mathcal{E}$ (denoted by $A\xleftrightarrow[]{\mathcal{E}}B$) if $A\cap B\neq \emptyset$, or $A\cap B=\emptyset$ while there exists a finite path $(x_0,x_1,...,x_n)$ such that $x_0\in A$, $x_n\in B$ and for each $0\le j\le n-1$, $\{x_j,x_{j+1}\}\in \mathcal{E}$.

\textbf{Phase transition and critical value of FRI:} For any $u,T>0$, we say $\mathcal{FI}^{u,T}$ percolates if with probability $1$, there exist some paths in $\mathcal{FI}^{u,T}$ which compose an infinite cluster (the word ``cluster'' means a connected subset of $\mathbb{L}^d$). Note that $\mathcal{FI}^{u,T}$ percolates if and only if $P\left(0 \xleftrightarrow[]{\mathcal{FI}^{u,T}} \infty\right)>0 $, where $\left\lbrace 0 \xleftrightarrow[]{\mathcal{FI}^{u,T}} \infty \right\rbrace$ is the event that $\mathcal{FI}^{u,T}$ contains an infinite cluster including the origin of $\mathbb{Z}^d$.

According to \cite{cai2021continuity,10.1214/21-ECP424}, the critical value of FRI is defined as: for any $T>0$, 
\begin{equation}
	u_*(T):=\sup \left\lbrace u>0:P\left( 0 \xleftrightarrow[]{\mathcal{FI}^{u,T}} \infty\right) =0 \right\rbrace.  
\end{equation}
In fact, Theorem 3.7 in \cite{10.1214/21-ECP424} has proved that $u_*(T)\in (0,\infty)$ for all $T>0$.

\section{Main results}

The main result proved in this paper is the accurate orders of the critical value $u_*(T)$:
\begin{theorem}\label{theorem1}
For $d\ge 3$, there exist $C_1(d)$ and $C_2(d)>0$ such that for all sufficiently large $T$, 
	\begin{itemize}
		\item when $d=3$, 
		\begin{equation}
			C_1T^{-\frac{1}{2}}  \le 	u_*(T)\le C_2T^{-\frac{1}{2}};
		\end{equation}
		
		\item  when $d=4$,  
		\begin{equation}
			C_1T^{-1}\cdot \log(T)  \le 	u_*(T)\le C_2 T^{-1}\cdot\log(T);
		\end{equation}
		
		\item  when $d\ge 5$, 
		\begin{equation}
		\label{d5}
			C_1 T^{-1}  \le 	u_*(T)\le C_2 T^{-1}.
		\end{equation}
	\end{itemize}
\end{theorem}

Note that \eqref{d5} was conjectured in \cite{cai2021non} for all $d\ge 3$. From the theorem above, we now see that it is the correct order only when $d\ge 5$. 
\begin{remark}
It was shown in \cite{cai2021non, 10.1214/21-ECP424}, that for all $d\ge 3$, $u_*(T)\sim T^{-1}$ as $T\to 0$, which combined with Theorem \ref{theorem1} implies that $u_*$ is globally inversely-proportional with respect to $T$ for all $d\ge 5$.
\end{remark}

We conjecture that the $C_1, C_2$ in Theorem \ref{theorem1} can actually be arbitrarily close to each other. I.e., 
\begin{conjecture}
There exist $C(d)\in(0,\infty), \ d\ge 3$ such that 
	\begin{itemize}
		\item when $d=3$, 
		\begin{equation}
			\lim_{T\to\infty}T^{\frac{1}{2}}\cdot u_*(T)= C(3);
		\end{equation}
		
		\item  when $d=4$,  
		\begin{equation}
			\lim_{T\to\infty}T \log^{-1}(T)\cdot  u_*(T)= C(4);
		\end{equation}
		
		\item  when $d\ge 5$, 
		\begin{equation}
		\label{d5}
			\lim_{T\to\infty} T\cdot 	u_*(T)= C(d).
		\end{equation}
	\end{itemize}
\end{conjecture}

At the same time, we recall the definition of the critical values in terms of $T$:
\begin{definition}{(Definition 2.3, \cite{10.1214/21-ECP424})}
	For $u>0$,  $d\ge 2$, we define that $$T_c^-(u,d):=\sup\{T_0>0:\forall 0<T<T_0,\mathcal{FI}^{u,T}\ does\ not\  percolate\} $$ and $$T_c^+(u,d):=\inf\{T_0>0:\forall T>T_0,\mathcal{FI}^{u,T}\ percolates\}.$$
\end{definition}
Then as a direct corollary of Theorem \ref{theorem1}, we may also have the following exact order of $T_c^\pm$, which shows that the upper bounds found in Corollary 3.6, \cite{10.1214/21-ECP424} are actually sharp:
\begin{corollary}
	When $d=3$, \begin{equation}
		\lim_{u\to 0}\frac{\log(T_c^-)}{-\log(u)}=	\lim_{u\to 0}\frac{\log(T_c^+)}{-\log(u)}= 2;
	\end{equation}
	When $d\ge 4$, \begin{equation}
		\lim_{u\to 0}\frac{\log(T_c^-)}{-\log(u)}=	\lim_{u\to 0}\frac{\log(T_c^+)}{-\log(u)}= 1.
	\end{equation}
\end{corollary}


The rest of this paper is outlined as follows: 
\begin{itemize}
\item The lower bound for $d=3$ is proved through a renormalization argument \cite{sznitman2012decoupling}. See Section \ref{section4.1}. 
\item For $d\ge 4$, the renormalization argument seems no longer applicable. Here we need to obtain the desired lower bounds by estimating the killed capacities of killed SRW trajectories. Intuitively, we prove that you a.s. cannot take infinitely many transfers over trajectories and the range you can reach is thus finite. Detailed proofs can be found in Section \ref{section4.2}-\ref{section4.3}. 
\item The upper bound estimate for all $d\ge 3$ is done in Section \ref{section_upperbound}, where the proof is based on a coarse-graining argument on a thick slab. Intuitively, one may consider certain ``good events" on each block within a two-dimensional slab whose thickness of order $T^{1/2}$, and show that the growth of such good events dominates a supercritical Bernoulli percolation. 
\end{itemize}

\section{Proof of the lower bound}

In this section, we prove the lower bound estimates in Theorem \ref{theorem1}. It is sufficient to prove that for any $d\ge 3$, there exists a constant $C_1(d)>0$ such that for all $T\ge 2$, $\mathcal{FI}^{C_1\cdot F_d\left(T\right)^{-1},T}$ does not percolate, where for any $a>1$, 
\begin{equation}
	F_d\left(a\right):=\left\{
	\begin{aligned}
		&a^{\frac{1}{2}} &\ \ \ \ & d=3; \\
		&\frac{a}{\log(a)} &\ \ \ \ & d=4; \\
		&a & \ \ \ \ & d\ge 5.
	\end{aligned}
	\right.
\end{equation}
We prove the case when $d=3$ in Section \ref{section4.1}, and $d\ge 4$ in Section \ref{section4.2}-\ref{section4.3}.


\subsection{Proof of the lower bound for $d=3$}\label{section4.1}
The lower bound estimate in 3-dimensional lattice is proved by following an approach of decomposition of connecting events, known as the ``renormalization scheme". Such technique has been widely applied in the researches of stochastic models such as Gaussian free field (\cite{popov2015decoupling,rodriguez2013phase}), random interlacements (\cite{sidoravicius2010connectivity}) and finitary random interlacements (\cite{cai2020chemical,cai2021continuity,10.1214/21-ECP424}). We hereby specify some additional notations.

\begin{enumerate}
	\item Let $L_0=C_0n_T$ and $l_0>100$ (the exact values of constants $C_0$ and $l_0$ will be determined later). For each $n\ge 1$, let $L_n=L_0l_0^{n}$ and $\mathbb{L}_n=L_n\mathbb{Z}^d$. 
	
	\item For $n\ge 0$ and $x\in \mathbb{Z}^d$, let $B_{n,x}=B_x\left(L_n\right) $ and $\widetilde{B}_{n,x}=\widetilde{B}_x\left(L_n\right)$.
	
	\item For $n\ge 0$, let $\mathcal{I}_n=\{n\}\times \mathbb{L}_n$. Then for each $n\ge 1$ and $(n,x)\in \mathcal{I}_n$, define that 
	\begin{equation}
		\mathcal{H}_1(n,x)= \left\lbrace (n-1,y)\in \mathcal{I}_{n-1}:B_{n-1,y}\subset B_{n,x}, B_{n-1,y}\cap \partial B_{n,x}\neq \emptyset  \right\rbrace, 
	\end{equation}
	\begin{equation}
		\mathcal{H}_2(n,x)= \left\lbrace (n-1,y)\in \mathcal{I}_{n-1}:B_{n-1,y}\cap \left\lbrace z\in \mathbb{Z}^d:d\left(z,B_{n,x}\right)= \lfloor \frac{L_n}{2} \rfloor  \right\rbrace\neq \emptyset \right\rbrace.  
	\end{equation}
	\item For $n\ge 0$ and $x\in \mathbb{Z}^d$, define 
	\begin{equation}\label{lambda}
		\begin{split}
			\Lambda_{n,x}=\Bigg\{\mathcal{T}\subset \bigcup_{k=0}^{n}\mathcal{I}_k:& \mathcal{T}\cap \mathcal{I}_n=(n,x)\ \text{and for all}\ (k,y)\in \mathcal{T}\cap \mathcal{I}_k, 0<k\le n\ \text{has}\\ &\text{two descendants}\ 
			(k-1, y_i(k,y))\in \mathcal{H}_i(k,y), i=1,2\ \text{s.t.}\\
			& \mathcal{T}\cap \mathcal{I}_{k-1}=\bigcup_{(k,y)\in \mathcal{T}\cap \mathcal{I}_{k}}\left\lbrace (k-1,y_1(k,y)),(k-1,y_2(k,y)) \right\rbrace     \Bigg\}. 
		\end{split}
	\end{equation}
	By (2.8) in \cite{rodriguez2013phase}, there exists $c_0(d)>0$ such that for all $n\ge 1$,\begin{equation}\label{4.5}
		|\Lambda_{n,x}|\le \left(c_0l_0^{2(d-1)} \right)^{2^n}.  
	\end{equation}
	
	\item For $n\ge 0$ and $x\in \mathbb{Z}^d$, we define the event 
	\begin{equation}
		F_{n,x}=\bigcap\limits_{\eta\in \mathcal{FI}^{C_1\cdot F_3(T)^{-1},T}:d\left(\eta(0), \widetilde{B}_{n,x}\right)> L_n} \left\lbrace V\left( \eta\right) \cap \widetilde{B}_{n,x}=\emptyset \right\rbrace. 
	\end{equation}
	
\end{enumerate}

\begin{figure}[h]
	\centering
	\includegraphics[width=0.55\textwidth]{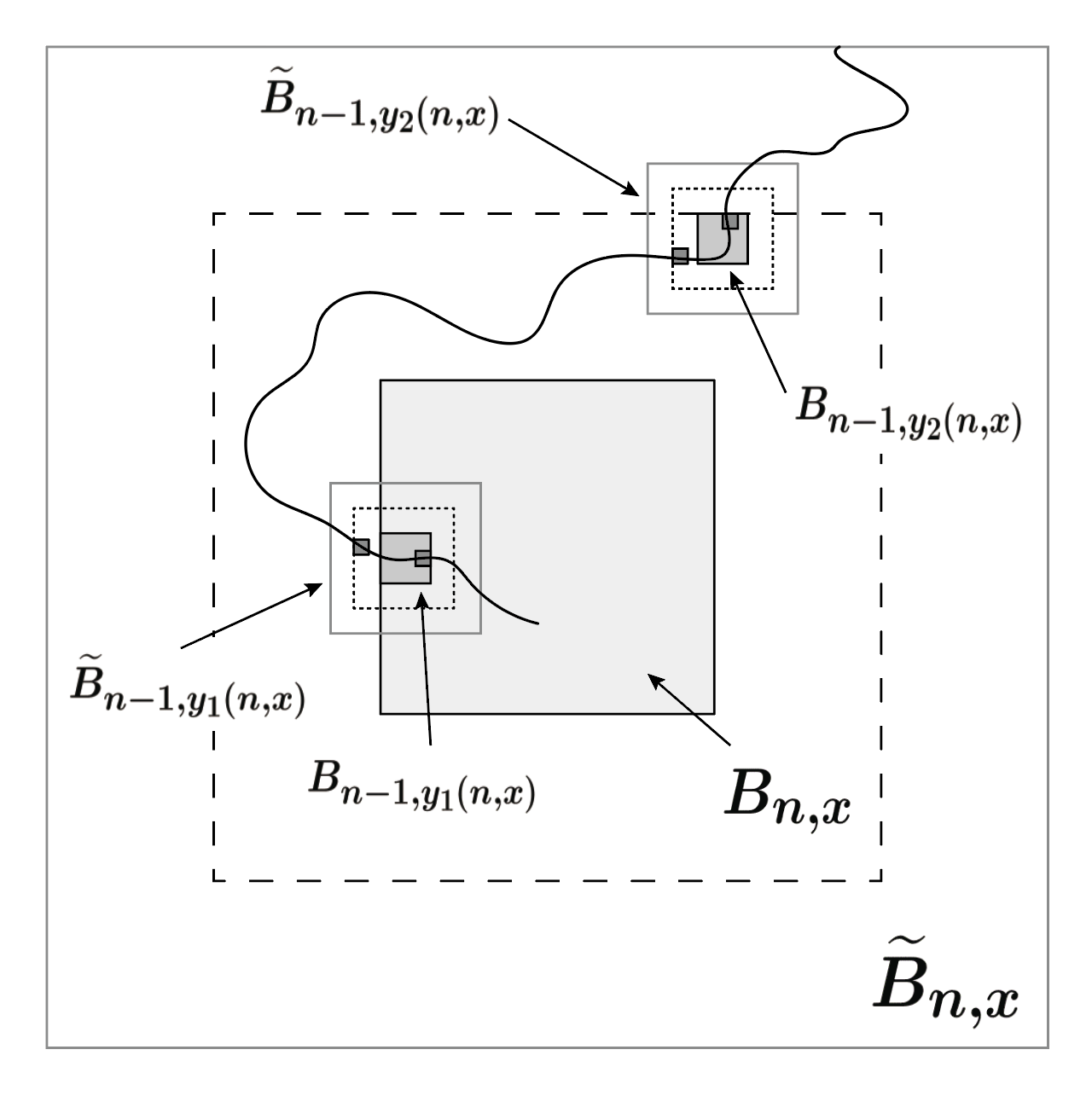}
	\caption{An	illustration of the renormalization scheme.}
	\label{figure1}
\end{figure}

To get an upper bound for the probability of event $F_{n,x}$, we need the following estimate of the diameter of a geometrically killed random walk:
\begin{lemma}\label{lemma_diameter}
	For $d\ge 3$, suppose that $ \left\lbrace X_i \right\rbrace_{i=0}^{N_T}$ is a geometrically killed random walk with law $P_0^{(T)}$. Then there exist $c_3(d),C_3(d)>0$ such that for any $T\ge 1$ and $t>0$, 
	\begin{equation}
		P\left(\max_{0\le i\le N_T}\left|X_i\right|\ge t\cdot n_T\right)\le C_3e^{-c_3t^{\frac{2}{3}}}.  
	\end{equation}
\end{lemma}
\begin{proof}[Proof of Lemma \ref{lemma_diameter}]
	Similar to (7.5) in \cite{10.1214/21-ECP424}, we have: for any $\delta\in (0,1)$ and $t>0$,  
	\begin{equation}\label{4.8}
		\begin{split}
			P\left(\max_{0\le i\le N_T}\left|X_i\right|\ge t\cdot n_T\right)\le& \left(\frac{T}{T+1}\right)^{t^{2-2\delta}n_T^2}+P\left(\max_{0\le i\le t^{2-2\delta}n_T^2}\left|X_i\right|\ge t\cdot n_T\right).
		\end{split}
	\end{equation}
	By Theorem 1.5.1 in \cite{lawler2013intersections} and (\ref{4.8}), 
	\begin{equation}\label{4.9}
		P\left(\max_{0\le i\le N_T}\left|X_i\right|\ge t\cdot n_T\right)\le e^{-c\cdot t^{2-2\delta}}+C(d)e^{-t^{\delta}}. 
	\end{equation}
	Hence, by taking $\delta=\frac{2}{3}$ in (\ref{4.9}), we have
	\begin{equation}
		P\left(\max_{0\le i\le N_T}\left|X_i\right|\ge t\cdot n_T\right)\le C'(d)e^{-c'(d)t^{\frac{2}{3}}}. 
	\end{equation} 
\end{proof}

Note that when event $F_{n,x}$ happens, the number of paths in $\mathcal{FI}^{C_1\cdot F_d(T)^{-1},T}$ that start from $\left\lbrace z\in \mathbb{Z}^d:d\left(z, \widetilde{B}_{n,x} \right)>L_n \right\rbrace $ and intersect $\widetilde{B}_{n,x}$ equals to 0. By Lemma \ref{lemma_diameter}, we get the following estimate for the event $F_{n,x}$: when $d=3$, for any $T\ge 1$, 
\begin{equation}\label{4.11}
	\begin{split}
		P\left[\left(F_{n,x} \right)^c \right]\le& 1-\exp(-\frac{6C_1F_{3}(T)^{-1}}{T+1}\sum_{z\in \mathbb{Z}^d:d\left(z, \widetilde{B}_{n,x} \right)>L_n }P_z^{(T)}\left(\left\lbrace X_i \right\rbrace_{i=0}^{N_T}\ \text{intersects}\ \widetilde{B}_{n,x} \right) )\\
		\le & \frac{6C_1F_{3}(T)^{-1}}{T+1}\sum_{z\in \mathbb{Z}^d:d\left(z, \widetilde{B}_{n,x} \right)>L_n }P_z^{(T)}\left(\max_{0\le i\le N_T}\left|X_i\right|\ge d\left(z, \widetilde{B}_{n,x} \right)\right)\\
		\le &\frac{6C_1F_{3}(T)^{-1}}{T+1} \sum_{m=L_n+1}^{\infty}C''m^{2}e^{-c_3\left(\frac{m}{n_T}\right)^{\frac{2}{3}} }.
	\end{split}
\end{equation}
Meanwhile, for any $k\ge n$, by the fact that $l_0^{\frac{2k}{3}}\ge l_0^{\frac{2n}{3}}+l_0^{\frac{2(k-n)}{3}}-1$, we have 
\begin{equation}\label{4.11.5}
	\begin{split}
		\sum_{m=L_k+1}^{L_{k+1}}m^{2}e^{-c_3\left(\frac{m}{n_T}\right)^{\frac{2}{3}} }\le& L_{k+1}\cdot L_{k+1}^{2}e^{-c_3C_0l_0^{\frac{2k}{3}}}\\
		\le &e^{c_3C_0}l_0^3\cdot n_T^3\cdot \left( l_0^{3n} e^{-c_3C_0l_0^{\frac{2n}{3}}}\right) \cdot \left( l_0^{3(k-n)} e^{-c_3C_0l_0^{\frac{2(k-n)}{3}}}\right). 
	\end{split}
\end{equation}
Combining (\ref{4.11}) and (\ref{4.11.5}), we have: for any $T\ge 1$,
\begin{equation}\label{4.11.6}
	\begin{split}
		P\left[\left(F_{n,x} \right)^c \right]\le& 6C_1C''e^{c_3C_0}l_0^3\cdot \left( \sum_{m=0}^{\infty} l_0^{3m} e^{-c_3C_0l_0^{\frac{2m}{3}}}\right) \cdot\left( l_0^{3n} e^{-c_3C_0l_0^{\frac{2n}{3}}}\right).
	\end{split}
\end{equation}
Choose a sufficiently large $l_0>0$ such that $l_0^{3n}<e^{0.5c_3l_0^{\frac{2n}{3}}}$ holds for all $n\ge 1$. Thus by (\ref{4.11.6}), for any $T\ge 1$, $C_0>1$ and $C_1<\left(6C''e^{c_3C_0}\sum_{m=0}^{\infty}l_0^{3m}e^{-c_3l_0^{\frac{2m}{3}}} \right)^{-1} $, we have 
\begin{equation}\label{PFnx}
	P\left[\left(F_{n,x} \right)^c \right]\le e^{-0.5c_3C_0l_0^{\frac{2n}{3}}}. 
\end{equation}

Now we are ready to prove the lower bound of Theorem \ref{theorem1} for $d=3$. The skills used in this proof are very similar to Proposition 5.1 in \cite{cai2020chemical} and Theorem 1 in \cite{cai2021continuity}:
\begin{proof}
	For any $x\in \mathbb{Z}^d$ and $n\ge 0$, we denote that $A_{n,x}=\left\lbrace B_{n,x}\xleftrightarrow[]{\mathcal{FI}^{C_1\cdot F_3(T)^{-1},T}}\widetilde{B}_{n,x}  \right\rbrace $. By (2.14) in \cite{rodriguez2013phase}, we have the following decomposition: 
	\begin{equation}\label{4.12}
		A_{n,x}\subset \bigcup_{\mathcal{T}\in \Lambda_{n,x}}A_{\mathcal{T}},
	\end{equation}
	where $A_{\mathcal{T}}=\bigcap\limits_{(0,y)\in \mathcal{T}\cap \mathcal{I}_0}A_{0,y}$. One may see Figure \ref{figure1} for an illustration of this decomposition. We consider the truncated event of $A_{n,x}$ as follows:
	\begin{equation}
		\hat{A}_{n,x}=\left\lbrace B_{n,x} \xleftrightarrow[]{\overline{\mathcal{FI}}} \widetilde{B}_{n,x} \right\rbrace,
	\end{equation}
where $\overline{\mathcal{FI}}=\sum\limits_{\eta\in \mathcal{FI}^{C_1\cdot F_3(T)^{-1},T}: d\left( \eta(0),\widetilde{B}_{n,x}\right)\le L_n} \delta_{\eta}$. By definition, we immediately have $\left( A_{n,x}\cap F_{n,x}\right) \subset \hat{A}_{n,x}\subset A_{n,x}$. 
	
	Recall the definition of $\Lambda_{n,x}$ in (\ref{lambda}). For each $\mathcal{T}\in \Lambda_{n,0}$, suppose that $\mathcal{T}\cap \mathcal{I}_{n-1}=\left\lbrace (n-1,x_1),(n-1,x_2) \right\rbrace $. Then $\mathcal{T}$ can be divided into two parts (denoted by $\mathcal{T}_1$ and $\mathcal{T}_2$) such that for any $i\in \{1,2\}$ and $(0,z)\in \mathcal{T}_i$, one has $z\in \widetilde{B}_{n-1,x_i}$. Noting that $d\left(\widetilde{B}_{n-1,x_1},\widetilde{B}_{n-1,x_2} \right)>10L_{n-1}$, we have that the events $\bigcap\limits_{(0,z)\in \mathcal{T}:z\in \widetilde{B}_{n-1,x_1}}\hat{A}_{0,z}$ and $ \bigcap\limits_{(0,z)\in \mathcal{T}:z\in \widetilde{B}_{n-1,x_2}} \hat{A}_{0,z}$ are independent. Therefore, by (\ref{PFnx}) we have 
	\begin{equation}\label{4.14}
		\begin{split}
			P\left(A_{\mathcal{T}}\right)
			=&P\left[ \left(\bigcap_{(0,z)\in \mathcal{T}:z\in \widetilde{B}_{n-1,x_1}}A_{0,z}\right) \cap \left( \bigcap_{(0,z)\in \mathcal{T}:z\in \widetilde{B}_{n-1,x_2}} A_{0,z}\right)\right] \\
			\le &  P\left[ \left(\bigcap_{(0,z)\in \mathcal{T}:z\in \widetilde{B}_{n-1,x_1}}\hat{A}_{0,z}\right) \cap \left( \bigcap_{(0,z)\in \mathcal{T}:z\in \widetilde{B}_{n-1,x_2}} \hat{A}_{0,z}\right)\right]\\
			&+P\left[ \left(F_{n-1,x_1} \right)^c\right] +P\left[ \left(F_{n-1,x_2}\right)^c\right]\\
			\le&P\left(\bigcap_{(0,z)\in \mathcal{T}:z\in \widetilde{B}_{n-1,x_1}}\hat{A}_{0,z}\right)\cdot P\left( \bigcap_{(0,z)\in \mathcal{T}:z\in \widetilde{B}_{n-1,x_2}} \hat{A}_{0,z}\right)+2e^{-0.5c_3C_0l_0^{\frac{2(n-1)}{3}}}\\
			\le &P\left(\bigcap_{(0,z)\in \mathcal{T}:z\in \widetilde{B}_{n-1,x_1}}A_{0,z}\right)\cdot P\left( \bigcap_{(0,z)\in \mathcal{T}:z\in \widetilde{B}_{n-1,x_2}} A_{0,z}\right)+2e^{-0.5c_3C_0l_0^{-\frac{2}{3}}\cdot 2^n}.
		\end{split}
	\end{equation}
	By induction and (\ref{4.14}), for each $n\ge 1$, 
	\begin{equation}\label{4.15}
		\begin{split}
			P\left(A_{\mathcal{T}}\right)+2e^{-0.5c_3C_0l_0^{-\frac{2}{3}}\cdot 2^n}
			\le & \left[P\left(\bigcap_{(0,z)\in \mathcal{T}:z\in \widetilde{B}_{n-1,x_1}}A_{0,z}\right)+2e^{-0.5c_3C_0l_0^{-\frac{2}{3}}\cdot 2^{n-1}}\right]\\
			& \cdot \left[P\left(\bigcap_{(0,z)\in \mathcal{T}:z\in \widetilde{B}_{n-1,x_1}}A_{0,z}\right)+2e^{-0.5c_3C_0l_0^{-\frac{2}{3}}\cdot 2^{n-1}}\right]\\
			\le &...\\
			\le &\left[P\left(A_{0,0}\right)+2e^{-0.5c_3C_0l_0^{-\frac{2}{3}}}\right]^{2^n}.  
		\end{split}
	\end{equation}
	
	For any $N>L_2$, let $n_0:=\max\left\lbrace m:L_m\le N \right\rbrace $. By (\ref{4.5}), (\ref{4.12}) and (\ref{4.15}), we have \begin{equation}\label{4.16}
		\begin{split}
			P\left(0\xleftrightarrow[]{\mathcal{FI}^{C_1\cdot F_3(T)^{-1},T}} \partial\left[-N,N \right) ^d   \right) 
			\le P\left(A_{n_0,0} \right)
			\le \left(c_0l_0^{2(d-1)} \right)^{2^n} \left[P\left(A_{0,0}\right)+2e^{-0.5c_3C_0l_0^{-\frac{2}{3}}}\right]^{2^n}. 
		\end{split}
	\end{equation}
	By (\ref{PFnx}) and the fact that $A_{0,0}\subset \hat{A}_{0,0}\cup \left( F_{0,0}\right)^c$, we have: for all $T\ge 1$,
		\begin{equation}\label{4.17}
			\begin{split}
				P\left(A_{0,0}\right)\le& P\left(\hat{A}_{0,0}\right)+P\left[ \left(F_{0,0}\right)^c\right] \\
				\le &P\left[ \exists\ \eta\in \mathcal{FI}^{C_1\cdot F_3(T)^{-1},T}\ \text{such that}\ d\left( \eta(0),\widetilde{B}_{0,0}\right)\le L_0  \right]+ e^{-0.5c_3C_0}\\
				\le  &1-\exp(-\frac{6C_1T^{-\frac{1}{2}}}{T+1}\cdot \left(3C_0n_T\right)^3 )+ e^{-0.5c_3C_0}\\
				\le &\frac{6C_1T^{-\frac{1}{2}}}{T+1}\cdot \left(3C_0n_T\right)^3+ e^{-0.5c_3C_0}\\
				\le &162C_1C_0^3+ e^{-0.5c_3C_0}.
			\end{split}
	\end{equation}
	Fix a sufficiently large $C_0$ such that $2e^{-0.5c_3C_0l_0^{-\frac{2}{3}}}+e^{-0.5c_3C_0}<0.25\left(c_0l_0^{2(d-1)} \right)^{-1}$ and then take some $C_1<\left(6C''e^{c_3C_0}\sum_{m=0}^{\infty}l_0^{3m}e^{-c_3l_0^{\frac{2m}{3}}} \right)^{-1}$ such that $162C_1C_0^3<0.25\left(c_0l_0^{2(d-1)} \right)^{-1}$. Thus, combining (\ref{4.16}) and (\ref{4.17}), we have 
		\begin{equation}
			P\left(0\xleftrightarrow[]{\mathcal{FI}^{C_1\cdot F_3(T)^{-1},T}} \partial\left[-N,N \right) ^d   \right)\le 2^{-2^{n_0}}, 
	\end{equation}
	which implies that $\mathcal{FI}^{C_1\cdot F_d\left(T\right)^{-1},T }$ does not percolate for all $T\ge 1$. 
\end{proof}

\subsection{Capacities of random walk trajectories}\label{section4.2}

Before turning to the proof in the case $d\ge 4$, we need to give some estimates on the expectations of capacities of random walk trajectories. 

We first cite a classical estimate for the expectations of capacities of simple random walk trajectories. 
\begin{lemma}[Theorem 2 in \cite{jain1968range}; Corollary 1.4, Proposition 1.5 in \cite{asselah2018capacity}]\label{lemma_cap}
	Suppose that $\left\lbrace X_i\right\rbrace_{i=0}^{\infty}$ is a simple random walk on $\mathbb{Z}^d$. For $d\ge 3$, there exist $C_4(d)$ and $C_5(d)>0$ such that for all $n\ge 2$, 
	\begin{equation}\label{new_34}
		C_4F_d(n)	\le 	E\left[ cap\left(\left\lbrace X_0,...,X_{n}\right\rbrace  \right)   \right]\le C_5F_d(n). 
	\end{equation}
\end{lemma}

Based on the lemma above, we show in this proof an upper bound for the expectations of  $T$-capacities of geometrically killed random walk trajectories with law $P_0^{(T)}$. To be precise, 
\begin{proposition}\label{lemma_capT}
	Suppose that $\left\lbrace X_i\right\rbrace_{i=0}^{N_T}$ is a geometrically killed random walk on $\mathbb{Z}^d$ with law $P_0^{(T)}$. For $d\ge 4$, there exists $C_6(d)>0$ such that for all $T\ge 2$, 
	\begin{equation}
		E\left[cap^{(T)}\left(\left\lbrace X_0,...,X_{N_T}\right\rbrace  \right)  \right]\le C_6F_d(T).  
	\end{equation}
\end{proposition}

\begin{remark}
	As a corollary of Lemma \ref{lemma_cap} and the fact that $cap^{(T)}(A)\ge cap(A)$ for all finite set $A\subset \mathbb{Z}^d$, there exists $C(d)>0$ such that $E\left[cap^{(T)}\left(\left\lbrace X_0,...,X_{N_T}\right\rbrace  \right)  \right]\ge C\cdot F_d(T)$ holds for all $d\ge 3$ and $T\ge 2$. Noting that the upper bound is straight, we also have the same order in (\ref{new_34}) for the expectations of $T$-capacities for geometrically killed random walks on $\mathbb{Z}^d$.
\end{remark}

In preparation for the proof of Lemma \ref{lemma_capT}, we cite two useful results on the intersecting probability of simple random walks as follows:
\begin{lemma}[Proposition 10.1.1, \cite{lawler2010random}]\label{lemma_lawler}
	Suppose that $X_\cdot$ is a simple random walk on $\mathbb{Z}^4$. For any $0\le m\le n< \infty$, let $X[m,n]=\left\lbrace X_j:m\le j\le n \right\rbrace $ and $X\left[ m,\infty\right) =\left\lbrace X_j:m\le j<\infty \right\rbrace$. Then there exists $C_7>0$ such that for all $n\ge 2$, 
	\begin{equation}
		P\left[X[0,n]\cap X\left[ 2n,\infty \right) \neq\emptyset \right]\le C_7 \left[ \log(n)\right] ^{-1}.
	\end{equation} 
\end{lemma}

\begin{lemma}[Theorem 4.4.1, \cite{lawler2013intersections}]\label{lemma_lawler2}
	Suppose that $X_\cdot$ and $X'_\cdot$ are two independent simple random walks starting from $0$ on $\mathbb{Z}^4$. Then there exists $C_8>0$ such that for all $n\ge2$, 
	\begin{equation}
		P\left( X[1,n]\cap X'[1,n]=\emptyset \right)\le C_8\left[ \log(n) \right] ^{-\frac{1}{2}}. 
	\end{equation}
\end{lemma}





\begin{proof}[Proof of Proposition \ref{lemma_capT}]
	For any $m\in \mathbb{N}^+$, we denote that $R_m=\left\lbrace X_0,...,X_{m}\right\rbrace$. 	When $d\ge 5$, by the subadditivity of $T$-capacity, we have: for any $T\ge 2$, 
	\begin{equation}
	\begin{split}
		E\left[cap^{(T)}\left(R_{N_T}\right)\right]\le cap^{(T)}\left(\{0\} \right)\cdot E\left[N_T+1 \right]\le cap^{(2)}\left(\{0\} \right)\cdot \left(T+1\right). 
	\end{split}
	\end{equation}
Thus, Proposition \ref{lemma_capT} holds for $d\ge 5$.

In the rest of this proof, we focus on the case when $d=4$. By definition,
	\begin{equation}\label{A6}
		\begin{split}
			&cap^{(T)}\left(R_{N_T}\right)-cap\left(R_{N_T}\right)\\
			= &\sum_{z\in R_{N_T}}\left(Es_{R_{N_T}}^{(T)}(z)-Es_{R_{N_T}}(z)\right)\\
			= &\sum_{z\in R_{N_T}}P_z^{(T)}\left[ \left\lbrace 1\le N_{T,z}'< \bar{H}_{R_{N_T}}\left(\hat{X}^z_{\cdot}\right)<\infty\right\rbrace \cup   \left\lbrace N'_{T,z}=0\right\rbrace \right] \\
			\le &\sum_{i=0}^{N_T}P\left(N_{T,X_i}'\le 2F_4(T)\right)+\sum_{i=0}^{N_T}P_{X_{i}}\left(2F_4(T)< \bar{H}_{R_{N_T}}\left(\hat{X}^{X_i}_\cdot\right) <\infty\right),
		\end{split}
	\end{equation}
		where $\left\lbrace \hat{X}^z_{\cdot}:z\in R_{N_T}\right\rbrace $ is a sequence of independent simple random walks with law $P_z$, $N_{T,z}'\overset{i.i.d.}{\sim} Geo\left(\frac{1}{T+1}\right)$ and all of them are independent to $X_\cdot$ and $N_T$. 
	
	For the first summation in the RHS of (\ref{A6}), we have 
	\begin{equation}\label{A..6}
		E\left[ \sum_{i=0}^{N_T}P\left(N_{T,i}'\le 2F_4(T)\right)\right]\le  E\left[(N_T+1)\cdot \frac{2F_4(T)+1}{T+1}\right]\le 3F_4(T).  
	\end{equation}
	
	For the second summation, we arbitrarily fix an $N_T\ge 2$. For any $0\le i\le N_T$, by the fact that $\bar{H}_{R_{N_T}}=\bar{H}_{R_{i}}\land \bar{H}_{X[i,N_T]}$, we have 
	\begin{equation}\label{A.3}
		\begin{split}
			&P_{X_{i}}\left( 2F_4(T)< \bar{H}_{R_{N_T}}\left(\hat{X}^{X_i}_\cdot \right) <\infty\right)\\
			\le &P_{X_{i}}\left( 2F_4(T)< \bar{H}_{R_{i}}\left(\hat{X}^{X_i}_\cdot \right)<\infty\right)+P_{X_{i}}\left(2F_4(T)< \bar{H}_{ X[i,N_T]}\left(\hat{X}^{X_i}_\cdot \right)<\infty\right).
		\end{split}
	\end{equation}

	We first estimate the expectation of first term in the RHS of (\ref{A.3}). When $0\le i\le m_T:=\lfloor2F_4(T)\rfloor$, we consider the random walk $Z_\cdot$ defined as follows: for $0\le j\le i$, $Z_j=X_j$; for $j\ge i+1$, $Z_j=\hat{X}^{X_i}_{j-i}$ (recall that $\hat{X}^{X_i}_\cdot$ is a simple random walk starting from $X_i$). Note that $Z_\cdot$ follows the law $P_{X_0}$. Let $m=\lfloor \frac{1}{2}\left(i+ m_T\right)\rfloor$. By Lemma \ref{lemma_lawler} and the fact that $ m\ge \lfloor F_4(T)\rfloor$, 
	\begin{equation}\label{A7}
		\begin{split}
			&E\left[P_{X_{i}}\left( 2F_4(T)< \bar{H}_{R_{i}}\left(\hat{X}^{X_i}_\cdot \right)<\infty\right)\right]\\
			\le  &E\left\lbrace  P_{X_0}\left(  Z\left[0,m \right]\cap Z\left[2m,\infty \right)\neq \emptyset  \right)   \right\rbrace
			\le \frac{C}{\log(T)}.
		\end{split}
	\end{equation}
	When $m_T<i\le N_T$, we denote that $r_0=0$, $s_1=r_1=\lfloor \frac{i+m_T}{2} \rfloor$ and for all $j\ge 2$, let $s_j=\lfloor \frac{i+m_T-r_{j-1}}{2} \rfloor$ and $r_j=r_{j-1}+s_j$. Let $p_0(i)=\min\left\lbrace l\ge 1:i-r_l \le  m_T  \right\rbrace $. Noting that  $r_{p_0+1}\ge i$ and for $1\le j\le p_0(i)+1$, $r_{j}+s_j\in \left\lbrace i+m_T,i+m_T-1 \right\rbrace $, 
	\begin{equation}\label{A9}
		\begin{split}
			&P_{X_{i}}\left( 2F_4(T)< \bar{H}_{R_{i}}\left(\hat{X}^{X_i}_\cdot \right)<\infty\right)\\
			= &P\left( Z[0,i]\cap \left[ i+m_T+1,\infty\right) \neq\emptyset,Z[0,i]\cap Z[i+1,i+m_T]=\emptyset \right)           \\
			\le &\sum_{j=1}^{p_0+1}P\left( Z[r_{j-1},r_{j}] \cap Z\left[ r_{j}+s_j,\infty\right)  \neq\emptyset,Z\left[i-m_T,i-1 \right]\cap Z\left[i+1,i+m_T \right] =\emptyset    \right). 
		\end{split}
	\end{equation}

	We estimate the probability in the RHS of (\ref{A9}) with an argument inspired by the proof of Proposition 10.1.1 in \cite{lawler2010random}. For $1\le j\le p_0-4$, let $Y^j:=\min_{r_{j-1}\le l\le r_j}\sum_{s=r_{j-1}}^{r_j}g\left(Z_{l},Z_{s} \right)$. Assume that $\gamma>0$ is a constant (we will determine it later). Then we define the event $A_j$ as 
	\begin{equation}\label{defAj}
		A_j:=\left\lbrace Y^j> \gamma \log(T) \right\rbrace.  
	\end{equation}
	In fact, events $A_j$ happens with a high probability. Since that for $1\le j\le p_0-4$,  $r_j-r_{j-1}\ge m_T$, we have $Y^j\ge \min\limits_{r_{j-1}\le  m\le r_{j}}\hat{Y}_m^j$, where $\hat{Y}_m^j:=\min\limits_{m\le l\le m+m_T}\sum_{s=m}^{m+m_T}g\left(Z_{l},Z_{s} \right)$. By Lemma 10.1.2 in \cite{lawler2010random}, there exists $\gamma>0$ such that for all $1\le j\le p_0-4$, 
	\begin{equation}\label{Aj}
		P\left[ \left(A_j\right)^c\right]\le P\left[\min_{r_{j-1}\le  m\le r_{j}}\hat{Y}_m^j\le \gamma \log(T) \right]    \le  \sum_{m=r_{j-1}}^{r_{j}}P\left[\hat{Y}^j_m\le \gamma \log(T)\right]\le C'N_T T^{-3}. 
	\end{equation}
	Note that $\left\lbrace Z[i-m_T,i-1]\cap Z[i+1,i+m_T]= \emptyset\right\rbrace $ and $A_j$ are measurable w.r.t. $Z\left[0,i+m_T\right]$. We denote by $\hat{P}$ be the law of $\left\lbrace Z_{k+1}-Z_{k}:k\ge i+m_T \right\rbrace$ and write the expectation under $\hat{P}$ as $\hat{E}$. Let $K^j=\sum_{l=r_{j-1}}^{r_{j}}\sum_{s=r_{j}+s_j}^{\infty}\mathbbm{1}_{Z_l=Z_s}$ be the number of intersecting. Note that $r_{j}+s_j\in \left\lbrace i+m_T,i+m_T-1 \right\rbrace $. Given $Z[0,i+m_T]$, if $r_{j}+s_j=i+m_T-1$ and $Z_{i+m_T-1}\in Z\left[r_{j-1},r_j \right]$, it is immediate that 
		\begin{equation}\label{new46}
		\hat{P}\left(K^j\ge 1\right)=1.
		\end{equation}
 Otherwise (i.e., $r_{j}+s_j=i+m_T$, or $r_{j}+s_j= i+m_T-1$ while $Z_{i+m_T-1}\notin Z\left[r_{j-1},r_j \right] $), 
	\begin{equation}\label{A.15}
		\hat{E}\left[K^j\right]=\hat{E}\left[\sum_{l=r_{j-1}}^{r_{j}}\sum_{s=i+m_T}^{\infty}\mathbbm{1}_{Z_l=Z_s}\right]=\sum_{l=r_{j-1}}^{r_{j}}  g\left(Z_l,Z_{i+m_T} \right).
	\end{equation}
Meanwhile, define that $k_j=\min\left\lbrace k\ge r_{j}+s_j:Z_k\in Z[r_{j-1},r_{j}] \right\rbrace $ and that $l_j=\min\left\lbrace r_{j-1}\le l\le r_{j}: Z_{l_j}=Z_{k_j} \right\rbrace $. By strong Markov property, 
\begin{equation}\label{A.16}
	\begin{split}
		\hat{E}\left[K^j\right]\ge &\hat{E}\left[K^j\cdot \mathbbm{1}_{K^j\ge 1 }\right]\\
		=&\hat{E}\left\lbrace E\left( K^j\big|Z\left[0,k_j \right] \right) \cdot \mathbbm{1}_{K^j\ge 1 }\right\rbrace\\
		\ge & \hat{E}\left[\mathbbm{1}_{K^j\ge 1 }\cdot \sum_{s=r_{j-1}}^{r_j}g\left(Z_s,Z_{l_j}\right) \right]\ge  Y^j\cdot\hat{P}\left(K^j\ge 1\right).
	\end{split}
\end{equation}	
Therefore, by (\ref{new46}), (\ref{A.15}) and (\ref{A.16}), 
\begin{equation}\label{A.17}
	\begin{split}
		\hat{P}\left(K^j\ge 1\right) 
		\le\frac{\sum_{l=r_{j-1}}^{r_{j}} g\left(Z_l,Z_{i+m_T} \right)}{\gamma \log(T)}+\sum_{l=r_{j-1}}^{r_{j}}\mathbbm{1}_{Z_l=Z_{i+m_T-1}}.
	\end{split}
\end{equation}
Combining (\ref{Aj}) and (\ref{A.17}), we have: for any $1\le j\le p_0-4$,
		\begin{equation}\label{A.18}
			\begin{split}
				&P\left( Z[r_{j-1},r_{j}] \cap Z\left[ r_{j}+s_j,\infty\right)  \neq\emptyset,Z\left[i-m_T,i-1 \right]\cap Z\left[i+1,i+m_T \right] =\emptyset    \right)\\
				\le &P\left(K^j\ge 1, Z[i-m_T,i-1]\cap Z[i+1,i+m_T]= \emptyset,A_j \right)
				+C'N_T\cdot T^{-3}\\
				\le & E\left[  \mathbbm{1}_{Z[i-m_T,i-1]\cap Z[i+1,i+m_T]= \emptyset}\left(\frac{\sum_{l=r_{j-1}}^{r_{j}} g\left(Z_l,Z_{i+m_T} \right)}{\gamma \log(T)}+\sum_{l=r_{j-1}}^{r_{j}}\mathbbm{1}_{Z_l=Z_{i+m_T-1}} \right) \right]\\
				 &+ C'N_T T^{-3}. 
			\end{split}
	\end{equation}
	Noting that for $1\le j\le p_0-4$, one has $s_j>4m_T$. Consider the revised random walk $W_\cdot$ starting from $Z_{i-m_T}$, which is defined as: for any $0\le l\le i-m_T$, $W_l:=Z_{i-m_T-l}$. By (4.2) in \cite{rath2010connectivity} and the fact that $i+m_T-r_j\ge s_j$, we have 
\begin{equation}\label{A17}
	\begin{split}
		&E\left(  \frac{\sum_{l=r_{j-1}}^{r_{j}} g\left(Z_l,Z_{i+m_T} \right)}{\gamma \log(T)}+\sum_{l=r_{j-1}}^{r_{j}}\mathbbm{1}_{Z_l=Z_{i+m_T-1}}\bigg| \left\lbrace Z_{k+1}-Z_k\right\rbrace_{i-m_T\le k\le i+m_T-1}\right)  \\
		\le &\left[\gamma \log(T)\right]^{-1}\cdot\sup_{z\in \mathbb{Z}^d}E\left(  \sum_{s=i-m_T-r_j}^{i-m_T-r_{j-1}}   g\left(W_s,z \right)\right)+ \sup_{w\in \mathbb{Z}^d}  \sum_{s=i-m_T-r_j}^{i-m_T-r_{j-1}} P\left(W_s=w \right)   \\
		\le &C''\left( \frac{s_j}{\log(T)\cdot \left( s_j-2m_T\right) }+\frac{s_j}{\left(s_j-2m_T\right)^2}\right) \le  \frac{3C''}{\log(T)}.
	\end{split}
\end{equation}
Since that the event $\left\lbrace Z[i-m_T,i-1]\cap Z[i+1,i+m_T]= \emptyset\right\rbrace $ is measurable w.r.t. $\left\lbrace Z_{k+1}-Z_k\right\rbrace_{i-m_T\le k\le i+m_T-1}$. By Lemma \ref{lemma_lawler2} and (\ref{A17}), we have 
	\begin{equation}\label{A18new}
		\begin{split}
			&E\left[  \mathbbm{1}_{Z[i-m_T,i-1]\cap Z[i+1,i+m_T]= \emptyset}\left(\frac{\sum_{l=r_{j-1}}^{r_{j}} g\left(Z_l,Z_{i+m_T} \right)}{\gamma \log(T)}+\sum_{l=r_{j-1}}^{r_{j}}\mathbbm{1}_{Z_l=Z_{i+m_T-1}} \right) \right]\\
			\le & \frac{3C''}{\log(T)}\cdot P\left(Z[i-m_T,i-1]\cap Z[i+1,i+m_T]= \emptyset\right)\\
			\le & \frac{3C''C_8}{\left[ \log(T)\right]^{\frac{3}{2}}}. 
		\end{split}
\end{equation}

	For $p_0-4<i\le p_0+1$, parallel to (\ref{A7}), we also have 
	\begin{equation}\label{A18}
		P\left( Z[r_{j-1},r_{j}] \cap Z\left[ r_{j}+s_j,\infty\right)  \neq\emptyset,Z\left[i-m_T,i-1 \right]\cap Z\left[i+1,i+m_T \right] =\emptyset    \right) \le \frac{C}{\log(T)}. 
	\end{equation}
	Noting that $i\le N_T$ and that $s_j\le \frac{s_{j-1}}{2}+1$, we have $p_0\le C'''\log(\log(N_T))$. Therefore, combining (\ref{A7}), (\ref{A9}), (\ref{A.18}), (\ref{A18new}) and (\ref{A18}), we have: for any $0\le i\le N_T$, 
		\begin{equation}\label{A.19}
			\begin{split}
				&E\left[ P_{X_{i}}\left( \frac{2T}{\log(T)}< \bar{H}_{R_{i}}\left(\hat{X}^{X_i}_\cdot \right)<\infty\right)\right]\\
				\le& C'''\log(\log(N_T)) \left[\frac{3C''C_8}{ \left[ \log(T)\right] ^{\frac{3}{2}}}+\frac{C'N_T}{T^3}\right]+\frac{5C}{\log(T)}.
			\end{split}
	\end{equation}

	For the second term in the RHS of (\ref{A.3}), by the reversibility of simple random walks, we can also prove the following inequality by using the same arguments: for any $0\le i\le N_T$, 
		\begin{equation}\label{A.20}
			\begin{split}
				&E\left[P_{X_{i}}\left(\frac{2T}{\log(T)}< \bar{H}_{ X[i,N_T]}\left(\hat{X}^{X_i}_\cdot \right)<\infty\right)\right]\\
				\le& C'''\log(\log(N_T)) \left[\frac{3C''C_8}{ \left[ \log(T)\right] ^{\frac{3}{2}}}+\frac{C'N_T}{T^3}\right]+\frac{5C}{\log(T)}.
			\end{split}
	\end{equation}

	By (\ref{A6}), (\ref{A..6}), (\ref{A.3}), (\ref{A.19}) and (\ref{A.20}), we have: for $d=4$ and $T\ge 2$, 
		\begin{equation}\label{A.21}
			\begin{split}
				&	E\left[\left( cap^{(T)}\left(R_{N_T}\right)-cap\left(R_{N_T}\right)\right)\cdot\mathbbm{1}_{N_T\ge 2} \right]\\
				\le &E\left[ 2N_T \left\lbrace C'''\log(\log(N_T)) \left[\frac{3C''C_8}{ \left[ \log(T)\right] ^{\frac{3}{2}}}+\frac{C'N_T}{T^3}\right]+\frac{5C}{\log(T)} \right\rbrace    \mathbbm{1}_{N_T\ge 2} \right]+3F_4(T) \\
				\le &\widetilde{C}\cdot F_4(T). 
			\end{split}
	\end{equation}
	Combine (\ref{A.21}) and Lemma \ref{lemma_cap},
		\begin{equation}
			\begin{split}
				&E\left[cap^{(T)}\left(R_{N_T}\right)\right]\\
				\le &	E\left[cap^{(T)}\left(R_{N_T}\right)\cdot\mathbbm{1}_{N_T\ge 2}\right]+ cap^{(2)}\left(\{0\}\right)\\
				=&E\left[\left( cap^{(T)}\left(R_{N_T}\right)-cap\left(R_{N_T}\right)\right)\cdot\mathbbm{1}_{N_T\ge 2} \right]+E\left[cap\left(R_{N_T}\right)\cdot\mathbbm{1}_{N_T\ge 2}\right]+ cap^{(2)}\left(\{0\}\right)\\
				\le &\widetilde{C}\cdot F_4(T)+E\left[C_5\cdot F_4(N_T)\cdot\mathbbm{1}_{N_T\ge 2}\right] + cap^{(2)}\left(\{0\}\right)\\
				\le &\hat{C}\cdot F_4(T). 
			\end{split}
	\end{equation}
	In conclusion, we now find the desired bound for $d=4$ and thus complete the proof of Proposition \ref{lemma_capT}.
\end{proof}

\subsection{Proof of the lower bound for $d\ge 4$}\label{section4.3}
With Proposition \ref{lemma_capT}, we are now ready to conclude the proof of the lower bound of Theorem \ref{theorem1} for $d\ge 4$ in this section. The Galton-Watson type approach here is, in some sense parallel to the argument in Section 3, \cite{MR3519252}.

We first introduce the following decomposition of clusters in $\mathcal{FI}^{u,T}$. For any subset $K\subset \mathbb{Z}^d$, we denote by $\Gamma^{u,T}_K$ the union of all clusters in $\mathcal{FI}^{u,T}$ which intersect $K$. Then we define the $k$-th layer of $\Gamma^{u,T}_K$ inductively as follows ($k\ge 1$):
\begin{enumerate}	
	\item Denote by $\Pi_1^{u,T}(K)$ the union of paths in $\mathcal{FI}^{u,T}$ intersecting $K$;

	\item For any $k\ge 2$, assume that we already have $\Pi_1^{u,T}(K),...,\Pi_{k-1}^{u,T}(K)$, then let $\Pi_k^{u,T}(K)$ be the union of paths in $\mathcal{FI}^{u,T}$ intersecting $ V\left( \Pi_{k-1}^{u,T}(K) \right) $ but not intersecting the set $K\cup \bigcup_{1\le j\le k-2} V\left(\Pi_j^{u,T}(K)\right)$.

\end{enumerate}
By definition, we immediately have the following two observations: (1) $\Gamma^{u,T}_K=\bigcup_{k=1}^{\infty}\Pi_{k}^{u,T}$; (2) if there exists some $k_0\ge 1$ such that $\Pi_{k_0}^{u,T}=\emptyset$, then for all $k\ge k_0+1$, $\Pi_{k_0}^{u,T}=\emptyset$.

Let $K=\{0\}\subset \mathbb{Z}^d$ ($d\ge 4$) and suppose that $T\ge 2$ and $v>0$. Note that when the event $\left\lbrace 0\xleftrightarrow[]{\mathcal{FI}^{v\cdot F_d(T)^{-1},T}} \infty \right\rbrace$ occurs, we have: for each $i\ge 1$,  $\Pi_i:=\Pi_i^{v\cdot F_d(T)^{-1},T}\left(K\right)\neq \emptyset$. Now we are going to show the following inequality by induction: there exists $C(d)>0$ such that for all $T\ge 2$, $v>0$ and $k\ge 1$,
\begin{equation}\label{424}
	E\left[cap^{(T)}\left( V\left( \Pi_k\right) \right) \right]\le C(d)\cdot \left(2dC_6v\right)^{k}, 
\end{equation}
where $C_6$ is the constant in Proposition \ref{lemma_capT}.

When $k=1$, note that the number of paths in $\mathcal{FI}^{v\cdot F_d(T)^{-1},T}$ intersecting $K$ is 
\begin{equation}
	N_1\sim Pois\left(2d\cdot cap^{(T)}\left(\{0\}\right)\cdot v\cdot F_d(T)^{-1} \right).  
\end{equation}
We denote these $N_1$ paths by $\eta_1^1,...,\eta^1_{N_1}$, then by Lemma \ref{lemma_capT}, 
\begin{equation}
	\begin{split}
			E\left[cap^{(T)}\left( V\left( \Pi_1\right) \right)\right]\le& E\left[\sum_{j=1}^{N_1}cap^{(T)}\left(V\left( \eta_j^1\right) \right)  \right]\\
			\le &E\left[N_1\right]\cdot C_6F_d(T) = cap^{(T)}\left(\{0\}\right)\cdot \left(2dC_6v \right).
	\end{split}
\end{equation}
Since for $T\ge 2$, $cap^{(T)}\left(\{0\}\right)\le C(d):=cap^{(2)}\left(\{0\}\right)$, we have that (\ref{424}) holds when $k=1$.

When $k\ge 2$, assume that (\ref{424}) holds for $k-1$. If $\Pi_1,...,\Pi_{k-1}$ is given, then by definition we have $\Pi_k$ is the union of paths in $\mathcal{FI}^{v\vdot F_d(T)^{-1},T}$ that intersect $V\left( \Pi_{k-1}\right) $ but does not intersect $K\cup \bigcup_{1\le j\le k-2}V\left( \Pi_{j}\right)$. Hence, we have 
\begin{equation}\label{4.27}
	\Pi_k\subset \hat{\Pi}_k:=\left\lbrace \eta\in \mathcal{FI}^{v\cdot F_d(T)^{-1},T}:\eta\ \text{intersects}\ V\left( \Pi_{k-1}\right) \right\rbrace.  
\end{equation}
Note that given $\Pi_{k-1}$, the number of paths in $\hat{\Pi}_k$ is a Poisson random variable  $Pois\left(2d\cdot cap^{(T)}\left( V\left( \Pi_{k-1}\right) \right)\cdot v\cdot F_d(T)^{-1} \right)$. Therefore, by Lemma \ref{lemma_capT},  (\ref{4.27}) and inductive hypothesis, we have 
\begin{equation}
	\begin{split}
		E\left[cap^{(T)}\left( V\left( \Pi_k\right) \right)\right]=& E\left\lbrace E\left[cap^{(T)}\left( V\left( \Pi_k\right) \right)\big|\Pi_j,1\le j\le k-1\right]   \right\rbrace \\
		\le & E\left\lbrace E\left[cap^{(T)}\left( V\left( \hat{\Pi}_k\right) \right)\bigg|\Pi_j,1\le j\le k-1\right]   \right\rbrace \\
		\le & E\left\lbrace 2d\cdot cap^{(T)}\left( V\left( \Pi_{k-1}\right) \right)\cdot v\cdot F_d(T)^{-1}\cdot C_6F_d(T) \right\rbrace\\
		\le & C\cdot \left(2dC_6v\right)^{k}. 
	\end{split}
\end{equation}	
In conclusion, 	(\ref{424}) holds for all $T\ge 2$, $v>0$ and $k\ge 1$. 

Take $C_1<\left(2dC_6\right)^{-1}$. Letting $v=C_1$ in (\ref{424}), by Markov inequality, we have: for any $T\ge 2$, 
\begin{equation}
	\begin{split}
		\sum_{k=1}^{\infty}P\left(\Pi_k\neq\emptyset \right)\le& \sum_{k=1}^{\infty}P\left(cap^{(T)}\left( V\left( \Pi_k\right) \right)\ge cap^{(T)}\left( \{0\}\right) \right)\\
		\le & \frac{1}{cap^{(T)}\left( \{0\}\right)}\sum_{k=1}^{\infty}E\left[cap^{(T)}\left( V\left( \Pi_k\right) \right)\right] \\
		\le & \frac{1}{cap^{(T)}\left( \{0\}\right)}\sum_{k=1}^{\infty} C\cdot \left(2dC_6C_1\right)^{k}<\infty. 
	\end{split}
\end{equation}  
Thus, by Borel-Cantelli Lemma, we have 
\begin{equation}
	P\left(0\xleftrightarrow[]{\mathcal{FI}^{C_1\cdot F_d(T)^{-1},T}} \infty\right)\le P\left( \left\lbrace \Pi_k\neq \emptyset\right\rbrace \ \text{i.o.}\right)=0,  
\end{equation}
which implies that $\mathcal{FI}^{C_1\cdot F_d\left(T\right)^{-1},T }$ does not percolate for all $T\ge 2$.   \qed

\section{Proof of the upper bound}\label{section_upperbound}

In this section, we prove the upper bound estimates in Theorem \ref{theorem1}. It is sufficient to prove that for $d\ge 3$, there exists $C_2(d)>0$ such that for all sufficiently large $T$, $\mathcal{FI}^{C_2\cdot F_d(T)^{-1},T}$ percolates. In fact, this result is moderately stronger than the previous result in Part 5 of Theorem 3, \cite{cai2021non}. To obtain this improvement, we adopt a completely different approach.

Here is a brief outline of our strategy. We explore the ``growth" of a certain FRI cluster according to a random algorithm $\mathbb{T}$, which looks for a fully covered box as a starting point. The cluster then grows step by step in a slab containing this box according to a coarse-graining procedure. As in \cite{grimmett1990supercritical}, we show in a  that this growth process dominates a supercritical Bernoulli percolation in $\mathbb{Z}^2$ and thus, produces an infinite cluster with a positive probability. We note that a very similar approach was also employed for the Wiener sausages in \cite{MR3519252}, while the criteria for ``good boxes" are substantially different.

To be precise, the random algorithm $\mathbb{T}$ is defined as follows:
\begin{enumerate}
		\item[Step 0:] Sample all the paths in $\mathcal{FI}^{C_2\cdot F_d(T)^{-1},T}$ starting from $\cup_{m=0}^{M(T)}B_{10n_T\cdot m e_1}^T$, where $M(T)$ is a function of $T$ which will be determined later, $e_1=(1,0,0,...,0)\in \mathbb{Z}^d$. 
		
		\begin{itemize}
			\item If there exists an integer $m\in \left[0,M(T)\right]$ such that all edges in $ B_{10n_T\cdot m e_1}^T$ are covered by the paths starting from $B_{10n_T\cdot m e_1}^T$, we denote by $m_0$ the lexicographically-smallest $m$ and let $x_0=m_0e_1$. Set that $\mathfrak{W}=\left\lbrace y\in \mathbb{Z}^d:y^{(1)}=m_0, \text{for all}\ 4\le i\le d,y^{(i)}=0\right\rbrace$, $\mathfrak{D}_0=\left\lbrace x_0 \right\rbrace $, $\mathfrak{E}_0=\emptyset$ and $\mathfrak{J}_0=\left\lbrace (x_0,B_{10n_T\cdot x_0}^{T}) \right\rbrace $.

			\item Otherwise, stop the algorithm.
		\end{itemize}

		\item[Step $k$:] Suppose that we already have $\mathfrak{D}_{k-1}$, $\mathfrak{E}_{k-1}$ and $\mathfrak{J}_{k-1}$ ($k\ge 1$). 
		\begin{itemize}
			\item If $\left( \partial^{out}_{\mathfrak{W}}\mathfrak{D}_{k-1}\right) \setminus \mathfrak{E}_{k-1}\neq \emptyset$ ($\partial^{out}_{\mathfrak{W}}A:=\left\lbrace x\in \mathfrak{W}\setminus A:\exists\ y\in A\ s.t.\ |x-y|_1=1 \right\rbrace $), we denote the lexicographically-smallest vertex in $\left( \partial^{out}_{\mathfrak{W}}\mathfrak{D}_{k-1}\right) \setminus \mathfrak{E}_{k-1}$ by $x_{k}$ and sample all the paths in $\mathcal{FI}^{C_2\cdot F_d(T)^{-1},T}$ starting from $B_{10n_Tx_{k}}$. By definition, there must exist at least one $(z,A_z)\in \mathfrak{J}_{k-1}$ such that $|z-x_k|_1=1$. And if such tuple is not unique, we choose $(z_{k-1},\mathcal{C}_{k-1})$ as the one such that $z_{k-1}$ is the lexicographically-smallest. 
			\begin{itemize}
				\item If there exists a sequence of paths in $\mathcal{FI}^{C_2\cdot F_d(T)^{-1},T}$ (denoted by $\{\eta_1,...,\eta_{m_{k}}\}$) starting from $B_{10n_Tx_{k}}^T$ satisfying the followings: 
					\begin{itemize}
					\item[(1)]$\text{cap}^{(T)}\left(V\left( \eta_1\cup...\cup\eta_{m_{k}} \right) \cap \widetilde{B}^T_{10n_Tx_{k}}\right) \ge c_1 n_T^{d-2}$ (recall the constant $c_1(d)$ in (\ref{cap_box}));
					
					\item[(2)]for any $1\le i\le m_{k}$, $\eta_i$ intersects $\mathcal{C}_{k-1}$,
					\end{itemize}
		   then let $\mathfrak{D}_{k}=\mathfrak{D}_{k-1}\cup \{x_{k}\}$, $\mathfrak{E}_{k}=\mathfrak{E}_{k-1}$ and $\mathfrak{J}_{k}=\mathfrak{J}_{k-1}\cup \left\lbrace \left( x_{k},V\left( \eta_1\cup...\cup\eta_{m_{k}}\right) \right) \right\rbrace $.

				\item Otherwise, let $\mathfrak{D}_{k}=\mathfrak{D}_{k-1}$, $\mathfrak{E}_{k}=\mathfrak{E}_{k-1}\cup \{x_{k}\}$ and $\mathfrak{J}_{k}=\mathfrak{J}_{k-1}$.  
			\end{itemize}

			\item If $\left( \partial^{out}_{\mathfrak{W}}\mathfrak{D}_{k-1}\right) \setminus \mathfrak{E}_{k-1}= \emptyset$, stop the algorithm. 
		\end{itemize}
		
\end{enumerate}

See Figure \ref{figure3} for an illustration of the random algorithm $\mathbb{T}$.

\begin{figure}[h]
	\centering
	\includegraphics[width=0.6\textwidth]{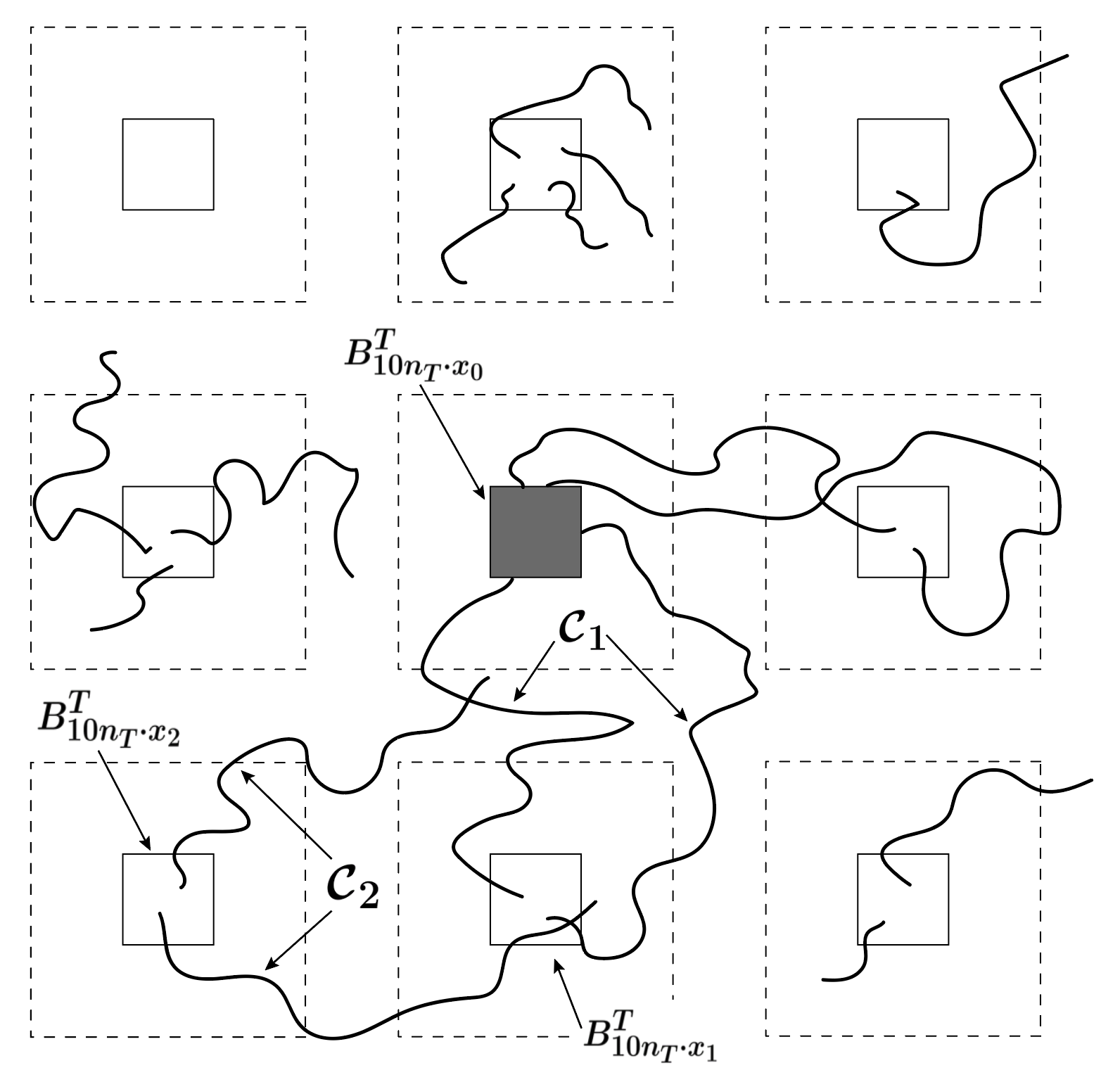}
	\caption{An	illustration of the random algorithm $\mathbb{T}$.}
	\label{figure3}
\end{figure}

	We claim that if the algorithm $\mathbb{T}$ has infinite steps, then $\mathcal{FI}^{C_2\cdot F_d(T)^{-1},T}$ contains an infinite cluster. To see this, suppose that the sequence of tuples added to  $\mathfrak{J}_\cdot$ is $\left\lbrace (z_i,A_i) \right\rbrace_{i=0}^{\infty}$. Then it is sufficient to confirm that all vertices in  $\bigcup_{i=0}^{\infty}A_i$ are connected to each other in $\mathcal{FI}^{C_2\cdot F_d(T)^{-1},T}$. We now prove this by induction: firstly, by the definition of Step $0$, it is immediate that all vertices in  $A_0=B_{10n_Tx_0}^{T}$ are connected; meanwhile, assuming that $\bigcup_{i=0}^{k}A_i$ is already connected, then by definition, all the vertices in $A_{k+1}$ are connected to $\bigcup_{i=0}^{k}A_i$ and thus $\bigcup_{i=0}^{k+1}A_i$ forms a cluster in $\mathcal{FI}^{C_2\cdot F_d(T)^{-1},T}$.

	Therefore, to prove $\mathcal{FI}^{C_2\cdot F_d(T)^{-1},T}$ percolates, it is sufficient to show that when $C_2$ is large enough, the growth of $\mathfrak{D}_k$ in $\mathbb{T}$ stochastically dominates a supercritical Bernoulli site percolation. I.e., the conditional probability (condition on the configuration of trajectories already sampled) of success (means that a new tuple is added to $\mathfrak{J}_\cdot$) of each $k$-th step ($k\ge 0$) is always greater than $p^+:=\frac{p_c(2)+1}{2}$, where $p_c(2)$ is the critical parameter of site percolation on $\mathbb{Z}^2$.

	Firstly, for Step 0, note that there exists $p_0(T)>0$ such that for all $x\in B_0^T$, $P_x^{(T)}\left(X_\cdot\ \text{covers all the edges in}\ B_0^{T}\right)\ge p_0(T)$. Hence, for each $0\le m\le M(T)$, the number of paths in $\mathcal{FI}^{C_2\cdot F_d(T)^{-1},T}$ that start from $B_{10n_T\cdot me_1}^T$ and cover all the edges in $B_{10n_T\cdot me_1}^T$ dominates $Pois(\hat{p}_0(T,C_2)):=Pois\left(\frac{2du}{T+1}\cdot C_2F_d(T)^{-1}\cdot n_T^d\cdot p_0(T)\right)$. Therefore, the probability of success of Step $0$ is greater than $1-e^{-\hat{p}_0(T)\cdot M(T,C_2)}$. Taking $M(T)=\frac{-\log(1-p^+)}{\hat{p}_0(T)}$, then for all large enough $T$, Step $0$ succeeds with probability at least $p^+$.

	For each $k$-th step ($k\ge 1$), we need the following lemma:
	\begin{lemma}\label{lemma_Pcap}
		Suppose that $A\subset \widetilde{B}_0^{T}$ satisfying $cap(A)\ge c_1n_T^{d-2}$. We denote by $\mathcal{G}^{\hat{u},T}$ the collection of paths in $\mathcal{FI}^{\hat{u}\cdot F_d(T)^{-1},T}$ starting from $B_{10n^{T}e_1}^{T}$ with length $\ge 2T$. Let $\mathcal{G}_A^{\hat{u},T}$ be the collection of paths in $\mathcal{G}^{\hat{u},T}$ intersecting $A$. Then there exist $C_9(d),c_4(d)>0$ such that for all $\hat{u}>0$ and sufficiently large $T$, 
		\begin{equation}\label{5.1}
			P\left[ cap\left(\widetilde{B}_{10n^{T}e_1}^{T}\cap V\left(\bigcup_{\eta\in \mathcal{G}_A^{\hat{u},T}} \eta \right)  \right)\ge c_1n_T^{d-2} \right]\ge 1-C_9e^{-c_4\hat{u}}.  
		\end{equation} 
	\end{lemma} 
	
	Once we have Lemma \ref{lemma_Pcap}, the existence of $C_2$ is immediate. To see this, take $A=\mathcal{C}_{k-1}$ in Lemma \ref{lemma_Pcap}. Note that when the event in (\ref{5.1}) happens, $\mathcal{G}_A^{\hat{u},T}$ meets all requirements for the Condition (1) and (2) in Algorithm $\mathbb{T}$ and thus the $k$-th step succeeds. Hence, it is sufficient to take $C_2>\frac{\log(C_9)-\log(1-p^+)}{c_4}$ and then the probability of success of each $k$-th step is always greater than $p^+$, for all large enough $T$. In conclusion, given Lemma \ref{lemma_Pcap}, we get the existence of $C_2$ and complete the proof of Theorem \ref{theorem1}.  \qed
	
	The proof of Lemma \ref{lemma_Pcap} is more standard and is put in the next subsection for completeness.

	\subsection{Proof of Lemma \ref{lemma_Pcap}}
	
	
	We hereby state a lemma parallel to Lemma 4 in \cite{rath2011transience}. In fact, a part of this lemma (in the case $d\ge 5$) is given by Lemma 4 in \cite{rath2010connectivity} and for $d=3,4$, this lemma can also be proved in the same way. Therefore, we just omit its proof here. 
	\begin{lemma}\label{lemma_capphi}
		For $d\ge 3$, let $X^{i}_\cdot $ be a sequence of independent simple random walks on $\mathbb{Z}^d$ with $X^{i}_0=x^i$. Define that  $\bar{\Phi}(N,n):=\bigcup_{i=1}^{N}\left\lbrace X^i_0,...,X^i_{\lfloor \frac{n^2}{2}\rfloor\land H_i} \right\rbrace $, where $H_i:=\min\left\lbrace {j:|X^i_j-x^i|\ge n}\right\rbrace$. Then there exists $c_5(d)>0$ such that for any integers $N\ge 1$ and $n\ge 2$, 
		\begin{equation}
			E\left[ \text{cap}\left(\bar{\Phi}(N,n)\right) \right] \ge  c_5\cdot \min\left\lbrace Nn\cdot F_d(n),n^{d-2} \right\rbrace.
		\end{equation}
	\end{lemma}

	Based on Lemma \ref{lemma_capphi}, we have the following estimate for the expectation of capacity of $\mathcal{G}_A^{\hat{u},T}$. 
	\begin{lemma}\label{lemma_Ecap_hatu}
		Suppose that $A\subset \widetilde{B}_0^{T}$ satisfying $cap(A)\ge c_1n_T^{d-2}$. We enumerate all the paths in $\mathcal{G}^{\hat{u},T}$ as $\eta_i$, $1\le i\le m^{\hat{u},T}$. For any $1\le i\le m^{\hat{u},T}$, let $H_i:=H_{\partial\widetilde{B}_{10n_Te_1}^T}\left(\eta_i\right)\land \lfloor T\rfloor$. Then there exists $c_6(d)>0$ such that for all $\hat{u}>0$ and sufficiently large $T$, 
		\begin{equation}
			E\left[cap\left( \bigcup_{1\le i\le m^{\hat{u},T}: H_A(\eta_i)\le 2T}\left\lbrace \eta_i(0),...,\eta_i(H_i) \right\rbrace \right)\right]\ge c_6\hat{u}\cdot  n_T^{d-2}.
		\end{equation}
	\end{lemma}
	\begin{proof}
		Similar to Lemma 3.1 in \cite{cai2020chemical}, we have: there exist $c(d),c'(d)>0$ such that for all large enough $T$, 
		\begin{equation}\label{5.6min}
			\min_{z\in \widetilde{B}_{10n_Te_1}^T} P_z\left(H_A\le T \right)\ge \frac{c\cdot cap(A)}{T^{\frac{d-2}{2}}} \ge c'.
		\end{equation} 
	
	For any $1\le i\le m^{\hat{u},T}$, we consider the following event:
		\begin{equation}
			D_i:=\left\lbrace \left\lbrace \eta_i(H_i),\eta_i(H_i+1),...,\eta_i(H_i+\lfloor T\rfloor ) \right\rbrace\cap A= \emptyset \right\rbrace.
		\end{equation}
Let $\left\lbrace U_i\right\rbrace_{i\ge 1} \overset{i.i.d.}{\sim} U[0,1]$, which is independent to the FRI $\mathcal{FI}$. Then we define a sequence of $0$-$1$ random variables $\left\lbrace V_i \right\rbrace_{i\ge 1}$ as follows: For any $1\le i\le m^{\hat{u},T}$, 
	\begin{equation}
		V_i:=\left\{
		\begin{array}{rcl}
		1,& \ \  &\text{if the events}\ D_i\ \text{and}\ \left\lbrace U_i\le \frac{c'}{P\left[ D_i\big|\eta_i\left(H_i\right) \right]  }\right\rbrace \text{both occur};      \\
		0,& \ \  &\text{otherwise}.  	    
		\end{array}	\right.
	\end{equation}
For any $i>m^{\hat{u},T}$, we set that $V_i:=\mathbbm{1}_{U_i\le c'}$. 

We claim that for each $1\le i\le m^{\hat{u},T}$, $V_i$ is independent to $\left\lbrace \eta_i(0),...,\eta_i(H_i) \right\rbrace$. Note that $V_i$ is measurable w.r.t. $\sigma\left(\eta_i(H_i),\left\lbrace \eta_i(j)-\eta_i(j-1) \right\rbrace_{H_i+1\le j\le H_i+\lfloor T\rfloor},U_i \right)$. Meanwhile, for any $x\in \widetilde{B}_{10n_Te_1}^T$, 
\begin{equation}
	\begin{split}
			P\left[V_i=1 \big| \eta_i(H_i)=x\right]
			=P\left[D_i\big| \eta_i(H_i)=x \right]\cdot P\left[U_i\le \frac{c'}{P\left[ D_i\big|\eta_i\left(H_i\right)=x \right]  }\right]=c'.
	\end{split}
\end{equation}
Combining these two observations, we know that $V_i$ is a $0$-$1$ random variable satisfying $P\left[V_i=1 \right]=c'$ and is independent to $\left\lbrace \eta_i(0),...,\eta_i(H_i) \right\rbrace$.

We denote by $E^U$ the expectation under the law of $\left\lbrace U_i\right\rbrace_{i\ge 1}$.  Noting that for any $1\le i\le m^{\hat{u},T}$, $\left\lbrace V_i=1 \right\rbrace\subset D_i\subset \left\lbrace H_A(\eta_i)\le 2T \right\rbrace$, we have 
\begin{equation}\label{new_72}
	\begin{split}
			&E\left[cap\left( \bigcup_{1\le i\le m^{\hat{u},T}: H_A(\eta_i)\le 2T}\left\lbrace \eta_i(0),...,\eta_i(H_i) \right\rbrace \right)\right]\\
			=&\left( E^U\times E\right) \left[cap\left( \bigcup_{1\le i\le m^{\hat{u},T}: H_A(\eta_i)\le 2T}\left\lbrace \eta_i(0),...,\eta_i(H_i) \right\rbrace \right)\right]\\
		\ge &\left( E^U\times E\right) \left[cap\left( \bigcup_{1\le i\le m^{\hat{u},T}: V_i=1}\left\lbrace \eta_i(0),...,\eta_i(H_i) \right\rbrace \right)\right].
	\end{split}
\end{equation}
By the law of total expectation, we have 
\begin{equation}\label{new_73}
	\begin{split}
		&\left( E^U\times E\right) \left[cap\left( \bigcup_{1\le i\le m^{\hat{u},T}: V_i=1}\left\lbrace \eta_i(0),...,\eta_i(H_i) \right\rbrace \right)\right]\\
		=&\left( E^U\times E\right)\left\lbrace E \left[cap\left( \bigcup_{1\le i\le m^{\hat{u},T}: V_i=1}\left\lbrace \eta_i(0),...,\eta_i(H_i) \right\rbrace \right)\Bigg|\left\lbrace V_i\right\rbrace_{i\ge 1}   \right]\right\rbrace. 
	\end{split}
\end{equation}
Let $E'$ be the expectation under an i.i.d. copy of $\mathcal{FI}$ (denoted by $\mathcal{FI}'$, which is also independent to $\left\lbrace U_i\right\rbrace_{i\ge 1}$). For the conditional expectation in the RHS of (\ref{new_73}), since that $\left\lbrace \left\lbrace \eta_i(0),...,\eta_i(H_i) \right\rbrace \right\rbrace_{1\le i\le m^{\hat{u},T}}$ is independent to $ \left\lbrace V_i \right\rbrace_{i\ge 1}$, we have 
\begin{equation}\label{new_74}
	\begin{split}
		&E \left[cap\left( \bigcup_{1\le i\le m^{\hat{u},T}: V_i=1}\left\lbrace \eta_i(0),...,\eta_i(H_i) \right\rbrace \right)\Bigg|\left\lbrace V_i\right\rbrace_{i\ge 1}   \right]\\
		=&E' \left[cap\left( \bigcup_{1\le i\le m^{\hat{u},T}: V_i=1}\left\lbrace \eta_i(0),...,\eta_i(H_i) \right\rbrace \right)\Bigg|\left\lbrace V_i\right\rbrace_{i\ge 1}   \right]. 
	\end{split}	
\end{equation}
Note that when taking the expectation of the RHS of (\ref{new_74}), we are actually deleting each path in $\mathcal{G}_A^{\hat{u},T}\left( \mathcal{FI}'\right)$ independently with probability $1-c'$. Therefore, by the thinning property of Poisson point process, it is equivalent to multiply the intensity measure of $\mathcal{FI}'$ by $c'$. To be precise, we have 
\begin{equation}\label{new_75}
	\begin{split}
	&\left( E^U\times E\right)\left\lbrace 	E' \left[cap\left( \bigcup_{1\le i\le m^{\hat{u},T}: V_i=1}\left\lbrace \eta_i(0),...,\eta_i(H_i) \right\rbrace \right)\Bigg|\left\lbrace V_i\right\rbrace_{i\ge 1}   \right]\right\rbrace\\
	=&E'\left[cap\left( \bigcup_{1\le i\le m^{c'\hat{u},T}}\left\lbrace \eta_i(0),...,\eta_i(H_i) \right\rbrace \right) \right]. 
	\end{split}
\end{equation}
In conclusion, combine (\ref{new_72})-(\ref{new_75}), 
\begin{equation}\label{new_76}
	\begin{split}
		&	E\left[cap\left( \bigcup_{1\le i\le m^{\hat{u},T}: H_A(\eta_i)\le 2T}\left\lbrace \eta_i(0),...,\eta_i(H_i) \right\rbrace \right)\right]\\
			\ge &E'\left[cap\left( \bigcup_{1\le i\le m^{c'\hat{u},T}}\left\lbrace \eta_i(0),...,\eta_i(H_i) \right\rbrace \right) \right].
	\end{split}
\end{equation}

		Meanwhile, by Lemma \ref{lemma_capphi} we have 
			\begin{equation}\label{5.9c_4min}
				\begin{split}
					E'\left[cap\left( \bigcup_{1\le i\le m^{c'\hat{u},T}}\left\lbrace \eta_i(0),...,\eta_i(H_i) \right\rbrace \right)\right]\ge c_5\cdot E'\left[\min \left\lbrace m^{c'\hat{u},T}n_T\cdot F_d(n_T),n_T^{d-2} \right\rbrace \right].
				\end{split}
		\end{equation}
		Note that $m^{c'\hat{u},T}\sim Pois\left(\frac{2dc'\hat{u}F_d(T)^{-1} n_T^d}{T+1} P_0^{(T)}\left(N_T\ge 2T \right)  \right)\ge Pois\left(c''(d)\hat{u}F_d(T)^{-1}n_T^{d-2}\right)$. By the large deviation bound for Poisson random variables (see (2.11) in \cite{rath2011transience} for example) and the fact that $2n_TF_d(n_T)\ge F_d(T) $ for all large enough $T$, we have 
		\begin{equation}\label{new_78}
			E'\left[\min \left\lbrace m^{c'\hat{u},T}n_T\cdot F_d(n_T),n_T^{d-2} \right\rbrace \right]\ge c'''(d)\hat{u}\cdot n_T^{d-2}. 
		\end{equation}
		Combining (\ref{new_76}), (\ref{5.9c_4min}) and (\ref{new_78}), then Lemma \ref{lemma_Ecap_hatu} follows. 
	\end{proof}

	Moreover, we also need some estimates on the second moments of capacities of simple random walk trajectories as follows:
	\begin{lemma}\label{lemma_cap2}
		Suppose that $X_\cdot$ is a simple random walk on $\mathbb{Z}^d$. Then for $d\ge 3$, there exists $C_{10}(d)>0$ such that for all sufficiently large $n$, 
		\begin{equation}\label{5.9_Ecap2}
			E\left\lbrace \left[ cap\left(X[0,n] \right) \right]^2\right\rbrace \le C_{10}\left[ F_d(n)\right]^2. 
		\end{equation}
	\end{lemma}
	\begin{proof}
		When $d=3$ or $d\ge 5$, the inequality (\ref{5.9_Ecap2}) is contained in the proof of Lemma 5 in \cite{rath2011transience}. 
		
		As for the case when $d=4$, by Theorem 1.1 in \cite{chang2017two}, one has 
		\begin{equation}\label{5.10_limitEcap}
			\lim\limits_{n\to \infty}	\frac{E\left\lbrace \left[ cap\left(X[0,n] \right) \right]^2\right\rbrace}{\left\lbrace E\left[ cap\left(X[0,n] \right)\right]  \right\rbrace^2 } =1.
		\end{equation}
		Combining (\ref{5.10_limitEcap}) and Lemma \ref{lemma_cap}, we have: for all large enough $n$, 
		\begin{equation}
			E\left\lbrace \left[ cap\left(X[0,n] \right) \right]^2\right\rbrace\le 2\left\lbrace E\left[ cap\left(X[0,n] \right)\right]  \right\rbrace^2\le 2\left[C_5F_4(n)\right]^2.
		\end{equation}
		In conclusion, the inequality (\ref{5.9_Ecap2}) holds for $d\ge 3$. 
	\end{proof}

	Now we are ready to prove Lemma \ref{lemma_Pcap}:
	
	\begin{proof}[Proof of Lemma \ref{lemma_Pcap}]
		It is sufficient to prove the following result: let $c(d)=2c_1\cdot c_6^{-1}$, then there exists $c'(d)>0$ such that for all large enough $T$, 
		\begin{equation}\label{5.9_Ecap}
			\begin{split}
				P\left[cap\left( \bigcup_{1\le i\le m^{c,T}: H_A(\eta_i)\le 2T}\left\lbrace \eta_i(0),...,\eta_i(H_i) \right\rbrace \right)\ge c_1n_T^{d-2}\right]\ge c'. 
			\end{split} 
		\end{equation}
		In fact, when (\ref{5.9_Ecap}) is given, we consider i.i.d. copies contained in $\mathcal{FI}^{\hat{u}\cdot F_d(T)^{-1},T}$ as follows: for $1\le j\le \lfloor c^{-1}\hat{u} \rfloor $, let $\mathcal{FI}_j:=\sum\limits_{(u_i,\eta_i)\in \mathcal{FI}^T:c(j-1)\cdot F_d(T)^{-1}<u_i\le cj\cdot F_d(T)^{-1} }\delta_{\eta_i}$. Then for each $1\le j\le \lfloor c^{-1}\hat{u} \rfloor$, by (\ref{5.9_Ecap}) we have
		\begin{equation}\label{5.10_Ecap}
			\begin{split}
				P\left[cap\left( \bigcup_{1\le i\le m^{\hat{u},T}: H_A(\eta_i)\le 2T,\eta_i\in \mathcal{FI}_j }\left\lbrace \eta_i(0),...,\eta_i(H_i) \right\rbrace \right)\ge c_1n_T^{d-2}\right]\ge c'. 
			\end{split} 
		\end{equation}
		Hence, by (\ref{5.10_Ecap}) and the independence betweem $\mathcal{FI}_j,1\le j\le \lfloor c^{-1}\hat{u} \rfloor$, we have 
		\begin{equation}
			\begin{split}
				&P\left[ cap\left(\widetilde{B}_{10n^{T}e_1}^{T}\cap V\left(\bigcup_{\eta\in \mathcal{G}_A^{\hat{u},T}} \eta \right)  \right)\ge c_1n_T^{d-2} \right]\\
				\ge &P\left[ cap\left(\bigcup_{1\le i\le m^{\hat{u},T}: H_A(\eta_i)\le 2T}\left\lbrace \eta_i(0),...,\eta_i(H_i) \right\rbrace\right)\ge c_1n_T^{d-2} \right]\\
				\ge &1- \prod_{j=1}^{\lfloor c^{-1}\hat{u} \rfloor}P\left[cap\left( \bigcup_{1\le i\le m^{\hat{u},T}: H_A(\eta_i)\le 2T,\eta_i\in \mathcal{FI}_j }\left\lbrace \eta_i(0),...,\eta_i(H_i) \right\rbrace \right)< c_1n_T^{d-2}\right]\\
				\ge &1-\left(1-c'\right)^{c^{-1}\hat{u}-1}.
			\end{split}
		\end{equation}
		In conclusion, (\ref{5.9_Ecap}) implies Lemma \ref{lemma_Pcap}. And thus we focus on the proof of (\ref{5.9_Ecap}).

		By the Paley-Zygmund inequality (i.e., for any non-negative random variable $Z$, $P\left(Z>\frac{1}{2}E[Z] \right)\ge \frac{\left( E[Z]\right) ^2}{4E\left[ Z^2\right]  } $) and Lemma \ref{lemma_Ecap_hatu},
		\begin{equation}\label{5.15}
			\begin{split}
				&P\left[cap\left( \bigcup_{1\le i\le m^{c,T}: H_A(\eta_i)\le 2T}\left\lbrace \eta_i(0),...,\eta_i(H_i) \right\rbrace \right)\ge c_1n_T^{d-2}\right]\\
				\ge &P\Bigg[cap\left( \bigcup_{1\le i\le m^{c,T}: H_A(\eta_i)\le 2T}\left\lbrace \eta_i(0),...,\eta_i(H_i) \right\rbrace \right)\\
				 &\ \ \ \ \ge \frac{1}{2}E\left\lbrace cap\left( \bigcup_{1\le i\le m^{c,T}: H_A(\eta_i)\le 2T}\left\lbrace \eta_i(0),...,\eta_i(H_i) \right\rbrace \right)\right\rbrace  \Bigg]\\   
				\ge &\frac{\left( c_6c\cdot n_T^{d-2}\right)^2}{4E\left\lbrace \left[cap\left( \bigcup\limits_{1\le i\le m^{c,T}: H_A(\eta_i)\le 2T}\left\lbrace \eta_i(0),...,\eta_i(H_i) \right\rbrace \right)\right]^2 \right\rbrace }.
			\end{split}
		\end{equation}
		Note that $m^{c,T}\sim Pois\left(\frac{2dcF_d(T)^{-1}n_T^d}{T+1} P_0^{(T)}\left(N_T\ge 2T \right)  \right)$ and $H_i\le \lfloor T \rfloor$. By the subadditivity of capacity, Lemma \ref{lemma_cap} and Lemma \ref{lemma_cap2}, for all large enough $T$, 
		\begin{equation}\label{5.16}
			\begin{split}
				&E\left\lbrace \left[cap\left( \bigcup_{1\le i\le m^{c,T}: H_A(\eta_i)\le 2T}\left\lbrace \eta_i(0),...,\eta_i(H_i) \right\rbrace \right)\right]^2 \right\rbrace\\
				\le &E\left\lbrace \left[\sum_{1\le i\le m^{c,T}}cap\left( \left\lbrace \eta_i(0),...,\eta_i(\lfloor T \rfloor ) \right\rbrace \right)\right]^2 \right\rbrace\\
				\le &E\left[m^{c,T}\right]\cdot  E\left\lbrace \left[cap\left(\left\lbrace \eta_i(0),...,\eta_i(\lfloor T \rfloor) \right\rbrace\right)  \right]^2 \right\rbrace\\
				& + E\left[\left( m^{c,T}\right)^2\right]\cdot  \left\lbrace E\left[cap\left(\left\lbrace \eta_i(0),...,\eta_i(\lfloor T \rfloor) \right\rbrace\right) \right]       \right\rbrace^2\\
				\le& C(d)F_d(T)^{-1}n_T^{d-2}\left[F_d(T)\right]^2+C'(d)F_d(T)^{-2}n_T^{2(d-2)}\left[F_d(T)\right]^2\\
				\le &C''(d)n_T^{2(d-2)}. 
			\end{split}
		\end{equation}
		Combining (\ref{5.15}) and (\ref{5.16}), then (\ref{5.9_Ecap}) follows and thus we get Lemma \ref{lemma_Pcap}. \end{proof}
	
\section*{Acknowledgment}
The authors would like to thank Drs. Xinyi Li, and Eviatar B. Procaccia for fruitful discussions. The authors gratefully thank Dr. Bal{\'a}zs R{\'a}th for pointing out to us the very important inferences \cite{MR3519252,worm2021} and for his very helpful comments.

\bibliographystyle{plain}
\bibliography{ref}
	
\end{document}